\documentclass[11pt]{article}       
\usepackage{geometry}               
\usepackage{graphicx}
\usepackage{caption}
\usepackage{subcaption}
\usepackage{comment}
\usepackage{amsmath} 
\usepackage{amssymb}  
\usepackage{amsthm}  
\usepackage{bm}
\usepackage{lipsum}
\usepackage{color}

\newtheorem{proposition}{Proposition}
\newtheorem{definition}{Definition}

\newtheorem{lemma}{Lemma}
\newtheorem{theorem}{Theorem}
\newtheorem{remark}{Remark}
\newtheorem{assumption}{Assumption}

\numberwithin{equation}{section}

\DeclareMathOperator*{\argmin}{arg\,min}

\title{Wasserstein Distributionally Robust Stochastic Control:\\ A Data-Driven Approach} 

\author{
 Insoon Yang\thanks{Department of Electrical and Computer Engineering, Automation and Systems Research Institute,  Seoul National University ({insoonyang@snu.ac.kr}). Supported in part by NSF under ECCS-1708906 and CNS-1657100,  Research Resettlement Fund for the new faculty of Seoul National University (SNU), 
 the  Creative-Pioneering Researchers Program through SNU, the Basic Research Lab Program through the National Research Foundation of Korea funded by the MSIT(2018R1A4A1059976), and Samsung Electronics.}
}

\date{}

\begin{document}
\maketitle

\pagestyle{myheadings}
\thispagestyle{plain}

\begin{abstract}
Standard stochastic control methods assume that the probability distribution of uncertain variables  is available. 
Unfortunately, in practice, obtaining accurate distribution information is a challenging task.  
To resolve this issue, we investigate the problem of designing a control policy that is robust against errors in the empirical distribution obtained from data.
This problem can be formulated as a two-player zero-sum dynamic game  problem, where the action space of the adversarial player  is a Wasserstein ball centered at the empirical distribution. We propose computationally tractable value and policy iteration algorithms with explicit estimates of the number of iterations required for constructing an $\epsilon$-optimal policy.
We show that the contraction property of associated Bellman operators extends a single-stage \emph{out-of-sample performance guarantee}, obtained using a measure concentration inequality, to the corresponding multi-stage guarantee without any degradation in the confidence level.
In addition, we characterize an explicit form of the optimal distributionally robust control policy and the worst-case distribution policy
for linear-quadratic problems with Wasserstein penalty.
Our study indicates that dynamic programming and Kantorovich duality play a critical role in solving and analyzing the Wasserstein distributionally robust stochastic control problems.
\end{abstract}


\section{Introduction}

The theory of stochastic optimal control is based on the assumption that the probability distribution of uncertain variables (e.g., disturbances) is fully known. However, this assumption is often restrictive in practice, because estimating an accurate distribution requires large-scale high-resolution sensor measurements over a long training period or multiple periods.
Situations in which uncertain variables are not directly observed are much more challenging; computational methods, such as filtering or statistical learning techniques, are often used to obtain the (posterior) distribution of the uncertain variables given limited observations. The accuracy of the obtained distribution is often unsatisfactory, as it is subject to the quality of the collected data, computational methods, and prior knowledge regarding the variables. If poor distributional information is employed in constructing a stochastic optimal controller, it does not guarantee optimality and can even cause catastrophic system behaviors~(e.g., \cite{Nilim2005, Samuelson2017}).

To overcome this issue of limited distribution information in stochastic control, we investigate a \emph{distributionally robust control} approach.
This emerging minimax stochastic control method minimizes a cost function of interest, assuming that the distribution of uncertain variables is not completely known, but is contained in a pre-specified \emph{ambiguity set} of probability distributions. 
In this paper, we model the ambiguity set as a statistical ball centered at an empirical distribution with a radius measured by  the \emph{Wasserstein metric}.
This modeling approach provides a straightforward means to incorporate data samples into distributionally robust control problems.
Our focus is to show that the resulting stochastic control problems have  
several salient features
in terms of computational tractability and out-of-sample performance guarantee.

Due to its superior statistical properties, 
the Wasserstein ambiguity set has recently received a great deal of attention in distributionally robust optimization~(e.g., \cite{Esfahani2015, Zhao2018, Gao2016, Blanchet2018}), learning~(e.g., \cite{Sinha2018, Chen2018}) and filtering~\cite{Shafieezadeh2018}. 
Specifically, the Wasserstein ball contains both continuous and discrete distributions while statistical balls with the $\phi$-divergence such as the Kullback-Leibler divergence centered at a discrete empirical distribution is not sufficiently rich to contain relevant continuous distributions. 
Furthermore, the Wasserstein metric addresses the closeness between two points in the support, unlike the $\phi$-divergence.
Due to the incapability of the $\phi$-divergence in terms of taking into account the distance between two support elements, the associated ambiguity set may contain irrelevant distributions~\cite{Gao2016}. 
For these reasons, we chose the Wasserstein metric to handle distribution ambiguity, 
although several other types of ambiguity sets have been proposed in the context of single-stage optimization
by using moment constraints~(e.g., \cite{Popescu2007, Delage2010, Zymler2013}), confidence sets~(e.g., \cite{Wiesemann2014}), and the $\phi$-divergences~(e.g., \cite{BenTal2013, Jiang2016}).

\subsection{Related Work}

Distributionally robust sequential decision-making problems have been studied in the context of finite Markov decision processes (MDPs) and continuous-state stochastic control.
In the finite MDP setting, 
dynamic programming (DP) approaches have been proposed~\cite{Xu2012, Yu2016, Yang2017lcss}. 
In~\cite{Xu2012}, moment-based ambiguity sets are used to impose constraints on the moments of distributions, such as mean and covariance.
This approach is further extended to  
handle more types of constraints, such as confidence sets and mean absolute deviation~\cite{Yu2016}, 
by using the lifting technique given in \cite{Wiesemann2014}.
Distributionally robust MDPs with Wasserstein balls are studied in~\cite{Yang2017lcss}, which provides computationally tractable reformulations and useful analytical properties. 

Continuous-state distributionally robust control problems can be considered as a class of minimax stochastic control on Borel spaces~\cite{Gonzalez2003}.
In the case of linear dynamics and quadratic cost functions, 
\cite{VanParys2016} focuses on linear policies and proposes tractable semidefinite program formulation when moment constraints are imposed.
A DP method is also proposed for moment-based ambiguity sets and applied to probabilistic safety specification problems~\cite{Yang2018}.
On the other hand, \cite{Tzortzis2019} uses a total variation ball to model distribution ambiguity and proposes a modified version of the classical policy iteration algorithm. 
Furthermore, a Riccati equation-based approach is also developed in the linear-quadratic regulator setting with the total variation ambiguity set~\cite{Tzortzis2016} and the relative entropy constraint~\cite{Petersen2000}.

%
%
%
%
%
%
%
%
%

\subsection{Contributions}

Departing from the aforementioned control approaches that indirectly use data samples, we consider \emph{continuous-state} distributionally robust control problems with Wasserstein ambiguity sets and develop a dynamic programming method to solve and analyze problems by directly using the data.
The following is a summary of the main contributions of this work.
First, we propose  computationally tractable
value and policy iteration algorithms 
with explicit estimates of the number of iterations necessary for obtaining an $\epsilon$-optimal policy. 
The original Bellman equation involves an infinite-dimensional minimax optimization problem, where the inner maximization problem is over probability measures in the Wasserstein ball. 
To alleviate the computational issue without sacrificing optimality, we 
reformulate Bellman operators by
using modern DRO based on Kantorovich duality~\cite{Esfahani2015, Gao2016}.
Second, we show that the resulting distributionally robust policy $\pi^\star$ has a probabilistic \emph{out-of-sample performance guarantee} by using the contraction property of associated Bellman operators and a measure concentration inequality.
In other words, when $\pi^\star$ is used,
 a probabilistic bound holds on the closed-loop performance evaluated
under a new set of samples that are selected independently of the training data. 
We observe that the contraction property of the Bellman operator seamlessly connects a single-stage performance guarantee to its multi-stage counterpart in a manner that is independent of the number of stages. 
Third, we consider a Wasserstein penalty problem and derive an explicit expression of the optimal control policy and the worst-case distribution policy, along with a Riccati-type equation in the linear-quadratic setting.
We also show that the resulting control policy converges to the optimal policy of the corresponding linear-quadratic-Gaussian (LQG) problem as the penalty parameter tends to $+\infty$.
The performance and utility of the proposed method are demonstrated through an investment-consumption problem and a power system frequency control problem.

This paper is significantly extended from its preliminary version~\cite{Yang2017cdc}, which models distribution ambiguity by using confidence sets.
Specifically, we consider Wasserstein ambiguity sets and investigate new salient features of the corresponding distributionally robust control framework such as $(i)$ a characterization of the worst-case distribution policy, $(ii)$ an out-of-sample performance guarantee, and $(iii)$ an explicit expression of the solution to linear-quadratic problems.

%
%

\subsection{Organization}

In Section~\ref{sec:setting}, we define optimal distributionally robust policies under ambiguous uncertainty and formulate the corresponding distributionally robust stochastic control problem as a dynamic game. 
In Section~\ref{sec:sol}, we develop a tractable semi-infinite program formulation of the Bellman equation and characterize one of the worst-case distribution policies by using Kantorovich duality.
In Section~\ref{sec:perf},
we examine
a probabilistic out-of-sample performance guarantee of the distributionally robust policy.
In Section~\ref{sec:pen},
we present the Wasserstein penalty problem and its explicit solution obtained from a Riccati-type solution. 
Finally, in Section~\ref{sec:exp}, we provide the results of our numerical experiments.

\subsection{Notation}

Given a Borel space $X$, we denote $\mathcal{P}(X)$ by the set of Borel probability measures on $X$. 
In addition, $\mathbb{B}_\xi (X)$  denotes the Banach space of measurable functions $v$ on $X$ with a finite  weighted sup-norm, i.e., 
$\| v \|_\xi := \sup_{\bm{x} \in X} (| v(\bm{x}) | / \xi (\bm{x}) ) < \infty$ given a measurable weight function $\xi : X \to \mathbb{R}$.
Let $\mathbb{B}_{lsc}(X)$ be the set of lower semicontinuous functions in $\mathbb{B}_\xi (X)$.

\section{Distributionally Robust Control of Stochastic Systems}
\label{sec:setting}

\subsection{Ambiguity in Stochastic Systems}

Consider a discrete-time stochastic system of the form
\begin{equation}
x_{t+1} = f (x_t, u_t, w_t),
\end{equation}
where $x_t \in \mathcal{X} \subseteq \mathbb{R}^n$ and $u_t \in \mathcal{U} \subseteq \mathbb{R}^m$ denote the system state and control input, respectively.  Here,  $w_t \in \mathcal{W} \subseteq \mathbb{R}^l$ is a  random disturbance. The probability distribution of $w_t$ is denoted by $\mu_t$. 
However, in practice,  the probability distribution is not fully known 
and is difficult to estimate accurately.
We assume that $\mathcal{X}$, $\mathcal{U}$ and $\mathcal{W}$
are Borel subsets of $\mathbb{R}^n$, $\mathbb{R}^m$ and $\mathbb{R}^l$, respectively.

Suppose that $w_t$'s are i.i.d. and that  we have  access to the sample $\{\hat{w}^{(1)}, \ldots, \hat{w}^{(N)}\}$ of $w_t$. 
One of the most straightforward approaches is to use the sample average approximation (SAA) method and solve the corresponding optimal control problem with the empirical distribution. 
This SAA-control problem can be formulated as
\begin{equation}\label{saa_opt}
\mbox{\small (SAA-control)} \; \inf_{\pi \in \Pi} \; \mathbb{E}^{\pi}_{w_t \sim \nu_N} \bigg [ \sum_{t=0}^\infty \alpha^t c(x_t, u_t) \mid x_0 = \bm{x} \bigg ],
\end{equation}
where 
$\nu_N$ denotes the empirical distribution constructed from the $N$-samples:
\begin{equation}\label{emp_dist}
\nu_N := \frac{1}{N} \sum_{i=1}^N \delta_{\hat{w}^{(i)}}
\end{equation}
with the Dirac delta measure $\delta_{\hat{w}^{(i)}}$ concentrated at $\hat{w}^{(i)}$.
Here, $\alpha \in (0,1)$ is a discount factor, $c: \mathcal{X} \times \mathcal{U} \to \mathbb{R}$ is a stage-wise cost function of interest,
and $\mathbb{E}_{w_t \sim \nu_N}^\pi$ denotes the expected value taken with respect to the probability measure induced by the control policy $\pi$ and the empirical distribution $\nu$. 
As the number  of samples, $N$, tends to infinity, the empirical distribution $\nu$ well approximates the true distribution $\mu$; thus, an optimal policy of the SAA-control problem presents a near-optimal performance.

Unfortunately, it takes a long simulation period or multiple episodes to obtain a large number of samples. Furthermore, in practice, it is likely that the sample data do not reflect the true distribution due to inaccurate sensor measurements or data corruption by malicious attackers (e.g., hackers). 
To resolve these issues in data-driven stochastic control, we propose an optimization method to construct a policy that is robust against errors in the empirical distribution~\eqref{emp_dist}.
More specifically, our policy minimizes the \emph{worst-case} total cost that is calculated under a probability distribution contained in a given set $\mathcal{D} \subset \mathcal{P}(\mathcal{W})$, which is called the \emph{ambiguity set} of probability distributions. 
The ambiguity set can be designed to adequately characterize errors in the empirical distribution.

\subsection{Distributionally Robust Policy} 

To formulate a concrete distributionally robust control problem, we consider a \emph{Markov (or stochastic) game with complete information} (e.g.,~\cite{Kuenle2000, Gonzalez2003}), which is a class of two-player zero-sum dynamic games:   Player~I (controller) determines a policy to minimize the total cost while   Player~II (adversary) selects the disturbance distribution $\mu_t$ of $w_t$  from  the ambiguity set $\mathcal{D}$ to maximize the same cost value.
Let $H_t$ be the set of \emph{histories} up to stage $t$, whose element is of the form $h_t := (x_0, u_0, \cdots, x_{t-1}, u_{t-1}, x_t)$.\footnote{All the results in this paper are valid with histories of the form $\tilde{h}_t := (x_0, u_0, w_0, \mu_0, \cdots, x_{t-1}, u_{t-1}, w_{t-1}, \mu_{t-1}, x_t)$ that also contains Player II's actions $(\mu_0, \cdots, \mu_{t-1})$; that is because under Assumption~\ref{ass_sc}, without loss of optimality, it suffices to focus on  stationary policies that depend only on current state information. We intentionally use the reduced version of histories, as
the realized distributions may not be observable in practice. }
The set of admissible control strategies (for Player I) is given by
$\Pi := \{ \pi := (\pi_0, \pi_1, \ldots) \: | \: \pi_t (\mathcal{U}(x_t) | h_t) = 1 \; \forall h_t \in H_t\}$, where $\pi_t$ is a stochastic kernel from $H_t$ to $\mathbb{R}^m$ and $\mathcal{U}(x_t) \subseteq \mathcal{U}$ is the set of admissible control actions (given that the system state is $x_t$ at stage $t$).
Similarly, the set of Player II's admissible strategies is defined by $\Gamma := \{\gamma := (\gamma_0, \gamma_1, \ldots ) \: | \: \gamma_t (\mathcal{D} | h_t^e) = 1 \; \forall h_t^e \in H_t^e\}$, where 
 $H_t^e$ is the set of \emph{extended histories} up to stage $t$, whose element is of the form $h_t^e := (x_0, u_0,  \mu_0, \cdots, x_{t-1}, u_{t-1},  \mu_{t-1}, x_t, u_t)$ and $\gamma_t$ is a stochastic kernel from $H_t$ to $\mathcal{P}(\mathcal{W})$.
Note that 
the ambiguity set $\mathcal{D}$ is the action space of Player II. 
Here, we allow Player II can change the distribution of $w_t$ over time. 
Thus, the strategy space for Player II is larger than necessary, and this gives an advantage to the adversary.
However, later we will show that an optimal policy of Player II is stationary under some assumption (see Proposition~\ref{prop:wd}).  

We consider the following infinite-horizon discounted cost function:
\begin{equation} \label{dr_opt}
\begin{split}
J_{\bm{x}}(\pi, \gamma) := \mathbb{E}^{\pi, \gamma} \bigg [
\sum_{t=0}^{\infty} \alpha^t c (x_t, u_t) \mid x_0 = \bm{x}
\bigg ],
\end{split}
\end{equation}
where  
 $\mathbb{E}^{\pi, \gamma}$ denotes expectation with respect to the probability measure induced by the strategy pair $(\pi, \gamma) \in \Pi \times \Gamma$.
 
Before defining a concrete stochastic  control problem, 
we impose the following standard assumption for measurable selection in semicontinuous models~\cite{Gonzalez2003}:
 \begin{assumption}\label{ass_sc}
Let $\mathbb{K} := \{  
(\bm{x}, \bm{u}) \in \mathcal{X} \times \mathcal{U} \mid
\bm{u} \in \mathcal{U}(\bm{x})
\}$.
\begin{enumerate}
\item The function $c$ is lower semicontinuous on $\mathbb{K}$, and 
\[
| c(\bm{x}, \bm{u}) | \leq b \xi (\bm{x}) \quad \forall (\bm{x}, \bm{u}) \in \mathbb{K},
\] 
for some constant $b \geq 0$ and continuous  function $\xi: \mathcal{X} \to [1, \infty)$ such that ${\xi}' (\bm{x}, \bm{u}) := \int_{\mathcal{W}} \xi(f(\bm{x}, \bm{u}, w))   \bm{\mu} (\mathrm{d}w)$ is continuous on $\mathbb{K}$ for any $\bm{\mu} \in \mathcal{D}$.
In addition, there exists a constant $\beta \in [1, 1/\alpha)$ such that
 ${\xi}'(\bm{x}, \bm{u}) \leq \beta \xi (\bm{x})$ for all $(\bm{x}, \bm{u}) \in \mathbb{K}$;

\item  For each continuous bounded function $\chi: \mathcal{X} \to \mathbb{R}$, the function ${\chi}' (\bm{x}, \bm{u}) := \int_{\mathcal{W}} \chi (f(\bm{x}, \bm{u}, w)) \bm{\mu} (\mathrm{d} w)$ is continuous on $\mathbb{K}$ for any $\bm{\mu} \in \mathcal{D}$;

\item The set $\mathcal{U}(\bm{x})$ is compact for every $\bm{x} \in \mathcal{X}$, and the set-valued mapping $\bm{x} \mapsto \mathcal{U}(\bm{x})$ is upper semicontinuous.
\end{enumerate}
\end{assumption}
The first condition trivially holds when $c$ is bounded. In fact, $\xi$ is a weight function introduced to relax the boundedness assumption. 
Assumption~\ref{ass_sc} ensures the existence of an optimal policy $\pi^\star$, which is deterministic and stationary, of a minimax control problem with the cost function~\eqref{dr_opt}~\cite[Theorem 4.1]{Gonzalez2003}.
Furthermore, the corresponding optimal value function lies in $\mathbb{B}_{lsc}(\mathcal{X})$ as discussed later.

We now define the \emph{optimal distributionally robust policies} as follows:
\begin{definition}
A control policy $\pi^\star \in \Pi$ is said to be an \emph{optimal distributionally robust policy} if it satisfies
\begin{equation} \label{eq:defn}
\sup_{\gamma \in \Gamma} \; J_{\bm{x}} (\pi^\star, \gamma) \leq \sup_{\gamma' \in \Gamma} \; J_{\bm{x}} (\pi, \gamma') \quad \forall \pi \in \Pi.
\end{equation}
\end{definition}
In words, an optimal distributionally robust policy achieves the minimal cost under the most adverse policies that select disturbance distributions in the ambiguity set $\mathcal{D}$.
Such a desirable policy can be obtained by solving the following problem:
\begin{equation}\label{dr_opt}
\begin{split}
\mbox{(DR-control)} \quad \inf_{\pi \in \Pi} \sup_{\gamma \in \Gamma} \; J_{\bm{x}}(\pi, \gamma),
\end{split}
\end{equation}
which we call the distributionally robust control (DR-control) problem.
The existence of an optimal policy under Assumption~\ref{ass_sc} will be  formalized in Theorem~\ref{thm:ds} in Section~\ref{sec:opt}.

The most important part of this formulation is the inner maximization problem over all disturbance distribution policies in $\Gamma$, which encodes distributional uncertainty through $\mathcal{D}$.
An optimal policy $\pi^\star$ has a performance guarantee in the form of an upper-bound, $\sup_{\gamma \in \Gamma} J_{\bm{x}} (\pi^\star, \gamma)$, if the ambiguity set is sufficiently large to contain the true distribution.
This performance guarantee may not be valid when a different control policy is used, as shown in~\eqref{eq:defn}.


\subsection{Wasserstein Ambiguity Set}

To complete the formulation of the DR-control problem, we consider a specific class of ambiguity sets using the Wasserstein metric.
Let $\mathcal{D}$ be a statistical ball centered at the empirical distribution $\nu_N$ defined by \eqref{emp_dist} with radius $\theta > 0$:
\begin{equation}\label{ball}
\mathcal{D} := \{
{\mu} \in \mathcal{P}(\mathcal{W}) \mid W_p({\mu}, \nu_N ) \leq \theta
\}.
\end{equation}
Here, the distance between the two probability distributions is measured by the Wasserstein metric of order $p \in [1, \infty)$,
\begin{equation}\label{wasserstein}
\begin{split}
W_p({\mu}, \nu_N) := \min_{\kappa \in \mathcal{P}(\mathcal{W}^2)}
\bigg\{
&\bigg [\int_{\mathcal{W}^2} d(w,  w')^p \: \kappa (\mathrm{d}w, \mathrm{d}w')\bigg ]^{\frac{1}{p}}  \mid \Pi^1 \kappa = {\mu}, \Pi^2 \kappa = \nu_N
\bigg\},
\end{split}
\end{equation}
where $d$ is a metric on $\mathcal{W}$, and $\Pi^i \kappa$ denotes the $i$th marginal of $\kappa$ for $i=1, 2$.
The Wasserstein distance between two probability distributions represents the minimum cost of transporting or redistributing mass from one to another via non-uniform perturbation, and the optimization variable $\kappa$ can be interpreted as a transport plan.

%
%
%
%

The minimization problem to identify an optimal transport plan $\kappa$  in \eqref{wasserstein} is called the \emph{Monge-Kantorovich problem}. 
The minimum of this problem can be found by solving the following dual problem:
\[
W_p(\mu, \nu_N)^p = \sup_{\varphi, \psi \in \Phi} \bigg [
\int_{\mathcal{W}} \varphi (w) \:  \mu (\mathrm{d} w)  + 
\int_{\mathcal{W}} \psi (w') \:\nu_N (\mathrm{d} w') 
\bigg ],
\]
where $\Phi := \{ 
(\varphi, \psi) \in L^1 (\mathrm{d} \mu) \times L^1(\mathrm{d} \nu_N) \mid
\varphi (w) + \psi (w') \leq d(w, w')^p \; \forall w, w' \in \mathcal{W}
\}$.
This equivalence is known as the \emph{Kantorovich duality principle}.
Then, the Wasserstein ball \eqref{wasserstein} can be expressed as follows:
\begin{lemma}\label{lem:ball}
The Wasserstein ambiguity set defined by \eqref{ball} is equivalent to
\begin{equation} \nonumber
\begin{split}
{\mathcal{D}} &= \bigg \{ {\mu} \in \mathcal{P} (\mathcal{W}) \mid 
\int_{\mathcal{W}} \varphi (w) \:  {\mu} (\mathrm{d} w) \: + \frac{1}{N}\sum_{i=1}^N \inf_{w \in \mathcal{W}} [ d(w, \hat{w}^{(i)})^p - \varphi (w) ]   \leq \theta^p \:\: \forall \varphi \in L^1(\mathrm{d}{\mu})
\bigg \}.
\end{split}
\end{equation}
\end{lemma}
A proof for this lemma is contained in Appendix~\ref{app:ball}.
Note that the minimization problem in the reformulated Wasserstein ball is finite dimensional, unlike the original Monge-Kantorovich problem.
In the following section, we propose computationally tractable value and policy iteration algorithms
 by using the reformulation results 
 in DRO based on Kantorovich duality.

\section{Dynamic Programming Solution and Analysis}
\label{sec:sol}

%

Our first goal is to 
develop a computationally tractable dynamic programming  (DP) solution for the DR-control problem \eqref{dr_opt}. 
We begin by characterizing an optimality condition using the Bellman's principle. 
%

\subsection{Bellman's Principle of Optimality}\label{sec:opt}

For any $v \in \mathbb{B}_\xi (\mathcal{X})$, let $T$ be the Bellman operator of the DR-control problem~\eqref{dr_opt}, defined by
\begin{equation}\nonumber
(T v) (\bm{x}) := \inf_{\bm{u} \in \mathcal{U}(\bm{x})} \sup_{\bm{\mu} \in \mathcal{D}} \bigg [ c (\bm{x}, \bm{u})
+ \alpha \int_{\mathcal{W}} v(f (\bm{x}, \bm{u}, w)) \bm{\mu}(\mathrm{d} w) \bigg ]
\end{equation}
for every $\bm{x} \in \mathcal{X}$.
Assumption~\ref{ass_sc} enables us to conduct the contraction analysis with respect to the weighted sup-norm $\| \cdot\|_\xi$ defined by
\[
\| v \|_\xi := \sup_{\bm{x} \in \mathcal{X}} \frac{|v(\bm{x}) |}{\xi(\bm{x})}.
\]
The second and third conditions in Assumption~\ref{ass_sc} play a critical role in preserving the lower semicontinuity of the value function when applying the Bellman operator
as well as
 in the existence and optimality of deterministic stationary policies.
Let $\Pi^{DS}$ be the set of deterministic stationary policies, i.e.,
$\Pi^{DS} := \{\pi: \mathcal{X} \to \mathcal{U} \mid \pi (x_t) = u_t \in \mathcal{U}(x_t)$, \mbox{$\pi$  measurable}\}.
Then, the following lemmas hold:

\begin{lemma}[Contraction and Monotonicity]~\label{lem:contraction}
Suppose that Assumption~\ref{ass_sc} holds.
Then, $Tv \in \mathbb{B}_{lsc}(\mathcal{X})$ for any $v \in \mathbb{B}_{lsc}(\mathcal{X})$.
Furthermore, the Bellman operator $T: \mathbb{B}_{lsc}(\mathcal{X}) \to \mathbb{B}_{lsc}(\mathcal{X})$ is a $\tau$-contraction mapping with respect to $\| \cdot \|_\xi$, where $\tau := \alpha\beta \in (0,1)$\footnote{Here, the constant $\beta \in [1, 1/\alpha)$ is defined in Assumption~\ref{ass_sc}-1).},
i.e., 
\[
\| Tv - Tv' \|_\xi \leq \tau \| v - v' \|_\xi \quad \forall v, v' \in \mathbb{B}_{lsc}(\mathcal{X}).
\]
Furthermore, $T$ is monotone, i.e.,
\[
T v \leq T v'  \quad \forall v, v' \in \mathcal{X}_\xi (\mathcal{X})\mbox{ s.t. } v \leq v'.
\]
\end{lemma}

\begin{lemma}[Measurable selection]\label{lem:ms}
Suppose that Assumption~\ref{ass_sc} holds.
There exist a measurable function $v^\star \in \mathbb{B}_{lsc}(\mathcal{X})$ and a deterministic stationary policy $\pi^\star \in \Pi^{DS}$ such that 
\begin{enumerate}
\item $v^\star$ is the unique function in $\mathbb{B}_{lsc}(\mathcal{X})$ that satisfies the following Bellman equation:
\begin{equation}\label{bellman}
v = Tv;
\end{equation}
\item given any fixed $\bm{x} \in \mathcal{X}$, 
\begin{equation} \nonumber
\begin{split}
&v^\star (\bm{x}) = \sup_{\bm{\mu} \in \mathcal{D}} \bigg [
 c (\bm{x}, \pi^\star (\bm{x})) + \alpha \int_{\mathcal{W}} v^\star (f (\bm{x}, \pi^\star (\bm{x}), w)) \: \bm{\mu}(\mathrm{d} w)
\bigg ]
\end{split}
\end{equation}
and
$\lim_{t \to \infty} \alpha^t \mathbb{E}^{\pi, \gamma} [ v^\star (x_t) ] = 0$ for all $(\pi, \gamma) \in \Pi \times \Gamma$.
\end{enumerate}
\end{lemma}
These lemmas follow immediately from~\cite[Lemma 4.4 and Theorem 4.1]{Gonzalez2003}.
In fact, for any $v \in \mathbb{B}_{lsc}(\mathcal{X})$,  there exists $\hat{\bm{u}} \in \mathcal{U}(\bm{x})$ such that
$(T v)(\bm{x}) = \sup_{\bm{\mu} \in \mathcal{D}} [
c (\bm{x}, \hat{\bm{u}}) + \alpha \int_{\mathcal{W}} v( f( \bm{x},  \hat{\bm{u}},  w)) \: \bm{\mu} (\mathrm{d} w) 
]$
for every $\bm{x} \in \mathcal{X}$ under Assumption~\ref{ass_sc} (see \cite[Lemma~3.3]{Gonzalez2003}).\footnote{Thus, the outer minimization problem in the definition of $T$ admits an optimal solution when $v \in \mathbb{B}_{lsc}(\mathcal{X})$, and ``$\inf$" can be replaced by ``$\min$."}
If we let $\pi^\star (\bm{x}) := \hat{\bm{u}}$ for each $\bm{x} \in \mathcal{X}$, then $\pi^\star$ is an optimal distributionally robust policy, which is deterministic and stationary. More specifically, the following principle of optimality holds: 
\begin{theorem}[Existence and optimality of deterministic stationary policy]\label{thm:ds}
Suppose that Assumption~\ref{ass_sc} holds.
Then, $(v^\star, \pi^\star) \in \mathbb{B}_{lsc}(\mathcal{X}) \times \Pi^{DS}$ defined in Lemma~\ref{lem:ms} satisfies
\[
v^\star(\bm{x}) = \inf_{\pi \in \Pi} \sup_{\gamma \in \Gamma} \; J_{\bm{x}} (\pi, \gamma) = \sup_{\gamma \in \Gamma} \; J_{\bm{x}} (\pi^\star, \gamma) \quad \forall \bm{x} \in \mathcal{X}.
\]
In  words, $v^\star$ is the optimal value function of the DR-control problem~\eqref{dr_opt}, and $\pi^\star$ is an optimal policy, which is
 deterministic and stationary.
\end{theorem}
The existence and optimality results are shown in a more general minimax control setting in \cite[Theorem 4.1]{Gonzalez2003}.

\subsection{Value Iteration}

To compute the optimal value function $v^\star$,
we first consider a \emph{value iteration} (VI) approach, $v_{k+1} := T v_k$, where $v_k$ denotes the value function evaluated at the $k$th iteration and $v_0$ is initialized as an arbitrary function in $\mathbb{B}_{lsc}(\mathcal{X})$. 
By the contraction property of $T$ (Lemma~\ref{lem:contraction}), 
the Banach fixed-point theorem implies that 
$v_k$ converges to $v^\star$ pointwise as $k$ tends to $\infty$ under Assumption~\ref{ass_sc}.
However, this approach requires us to solve the infinite-dimensional minimax optimization problem in the Bellman operator for each $\bm{x} \in \mathcal{X}$ in each iteration. 
To alleviate this issue, 
we reformulate the problem into a computationally tractable form
 by using modern Wasserstein DRO~\cite{Esfahani2015, Gao2016}.

\begin{proposition}\label{prop:si}
Suppose that the function $w \mapsto v(f(\bm{x}, \bm{u}, w))$ lies in $L^1 (\mathrm{d} \nu_N)$ for each $(\bm{x}, \bm{u}) \in \mathbb{K}$.
Then, the Bellman operator $T$ can be expressed as
\begin{equation} \label{semi}
\begin{split}
(T  v)(\bm{x}) = 
\inf_{\bm{u}, \lambda, \ell} \; &\bigg [ \lambda \theta^p + c (\bm{x}, \bm{u}) +  \frac{1}{N} \sum_{i=1}^N \ell_i \bigg ]\\
\mbox{s.t.} \; &  \alpha v( f(\bm{x}, \bm{u}, w)) - \lambda d( w, \hat{w}^{(i)} )^p \leq  \ell_i \; \; \forall w \in \mathcal{W}\\
&  \bm{u} \in \mathcal{U}(\bm{x}), \: \lambda \geq 0, \: \ell \in \mathbb{R}^N
\end{split}
\end{equation}
for each $\bm{x} \in \mathcal{X}$,
where the first inequality constraint holds for all $i=1, \ldots, N$. 
\end{proposition}
This reformulation can be obtained by using Kantorovich duality  on the Wasserstein ambiguity set (Lemma~\ref{lem:ball}).
It is shown in \cite[Theorem 1]{Gao2016} that there is no duality gap.

Note that 
the reformulated optimization problem in Proposition~\ref{prop:si}  has finite-dimensional decision variables as $\bm{u} \in \mathcal{U}(\bm{x}) \subseteq \mathcal{U} \subseteq \mathbb{R}^m$, $\lambda \in \mathbb{R}$ and $\ell \in \mathbb{R}^N$.
However, the first inequality constraint must hold for all $w$ in the support $\mathcal{W}$, which could be a dense set. 
Thus, in general, the reformulated problem is a \emph{semi-infinite program}. 
This semi-infinite program can be solved 
by using
several existing convergent algorithms, such as discretization, sampling-based methods~(see \cite{Reemtsen1991, Hettich1993, Lopez2007, Calafiore2005} and the references therein).

To interpret this reformulation, we consider the following equivalent integral form: 
\begin{equation} \nonumber
\begin{split}
&(T  v)(\bm{x}) = \inf_{\bm{u} \in \mathcal{U}(\bm{x}), \lambda\geq 0}   \bigg [ \lambda \theta^p  +  \int_{\mathcal{W}}  \sup_{w \in \mathcal{W}}  \big [ c(\bm{x}, \bm{u}) + \alpha v (f(\bm{x}, \bm{u}, w)) - \lambda d(w,  {w}' )^p \big ] \nu_N (\mathrm{d} w') \bigg ].
\end{split}
\end{equation}
The integrand above can be interpreted as a \emph{regularized cost-to-go} function.
The regularized value is then integrated using the empirical distribution $\nu_N$. 
The first term $\lambda \theta^p$, which is nonnegative, is added to compensate for this regularization effect and the optimism induced by the empirical distribution so that the reformulated optimization problem is consistent with the original one.

We define an \emph{$\epsilon$-optimal policy} of \eqref{dr_opt} as $\pi_\epsilon \in \Pi$ that satisfies
\[
\| v^{\pi_\epsilon} - v^\star\|_\xi < \epsilon
\]
for $\epsilon > 0$,
where $v^\pi: \mathcal{X} \to \mathbb{R}$ is the (worst-case) value function of a policy $\pi \in \Pi$, i.e.,
\begin{equation}\label{worst_value}
v^\pi (\bm{x}) := \sup_{\gamma \in \Gamma} J_{\bm{x}} (\pi, \gamma).
\end{equation}
The following VI algorithm can be used to find an $\epsilon$-optimal policy: 
\begin{enumerate}
\item Initialize $v_0$ as  an arbitrary function in $\mathbb{B}_{lsc} (\mathcal{X})$, and set $k := 0$;

\item For each $\bm{x} \in \mathcal{X}$,
compute 
\[
v_{k+1} (\bm{x}):= (T v_k) (\bm{x})
\]
 by
 solving the semi-infinite program~\eqref{semi}  with $v:= v_k$;

\item If the stopping criterion is met, then go to Step 4); Otherwise, set $k \leftarrow k+1$ and go to Step 2);

\item For each $\bm{x} \in \mathcal{X}$, set 
\[
\hat{\pi} (\bm{x}) := \hat{\bm{u}},
\]
 where $\hat{\bm{u}}$ is an optimal $\bm{u}$ of the semi-infinite program~\eqref{semi} that computes $(Tv_{k}) (\bm{x})$, and stop.

\end{enumerate}
Note that 
the existence of an optimal $\hat{\bm{u}}$ in Step 4) is guaranteed under Assumption~\ref{ass_sc} by \cite[Lemma 3.3]{Gonzalez2003}.
A typical stopping criterion in VI is $\| v_{k+1} - v_k \|_\xi  < \delta$ for some threshold $\delta > 0$.
However, 
we can even compute the number of iterations required to achieve the desired precision $\epsilon > 0$. 
Given any $\pi \in \Pi^{DS}$ and $v \in \mathbb{B}_\xi(\mathcal{X})$, 
let 
\[
(T^\pi v)(\bm{x}) := \sup_{\bm{\mu} \in \mathcal{D}} \bigg [
c (\bm{x}, \pi(\bm{x}))
+ \alpha \int_{\mathcal{W}} v(f (\bm{x}, \pi(\bm{x}), w)) \bm{\mu}(\mathrm{d}w)
\bigg]
\]
for all $\bm{x} \in \mathcal{X}$. 
The Bellman operator $T^\pi$ has the following properties:
\begin{lemma}\label{lem:cont2}
Suppose that Assumption~\ref{ass_sc} holds. 
Then, given any $\pi \in \Pi^{DS}$, we have $T^\pi v \in \mathbb{B}_{\xi}(\mathcal{X})$ for any $v \in \mathbb{B}_{\xi}(\mathcal{X})$.
Furthermore, 
the operator $T^\pi:  \mathbb{B}_{\xi} (\mathcal{X}) \to  \mathbb{B}_{\xi} (\mathcal{X})$ is a $\tau$-contraction mapping with respect to $\| \cdot \|_\xi$, i.e.,
\[
\|T^\pi v - T^\pi v'\|_\xi \leq \tau \| v - v'\|_\xi \quad \forall v, v' \in \mathbb{B}_{\xi} (\mathcal{X}),
\]
where $\tau :=\alpha\beta \in (0,1)$. 
Furthermore, $T^{\pi}$ is monotone, i.e.,
\[
T^{\pi} v \leq T^{\pi} v'  \quad \forall v, v' \in \mathcal{X}_\xi (\mathcal{X})\mbox{ s.t. } v \leq v'.
\]
\end{lemma}
\begin{proof}
By Assumption~\ref{ass_sc}, 
it is clear that $T^\pi v \in \mathbb{B}_\xi (\mathcal{X})$ if $v \in \mathbb{B}_\xi (\mathcal{X})$. 
Fix arbitrary $v, v'  \in \mathbb{B}_{\xi} (\mathcal{X})$, and an arbitrary $\bm{x} \in \mathcal{X}$.
For any $\epsilon > 0$, there exists $\hat{\bm{\mu}} \in \mathcal{D}$ such that
\[
(T^\pi v)(\bm{x}) - \epsilon < c(\bm{x}, \pi (\bm{x})) + \alpha \int_\mathcal{W} v(f (\bm{x}, \pi(\bm{x}), w))  \hat{\bm{\mu}}(\mathrm{d}w).
\]
Thus, we have
\begin{equation}\nonumber
\begin{split}
(T^\pi v)(\bm{x}) - (T^\pi v')(\bm{x}) - \epsilon
& < \alpha \int_\mathcal{W} [v(f (\bm{x}, \pi(\bm{x}), w)) - v'(f(\bm{x}, \pi(\bm{x}), w))]  \hat{\bm{\mu}}(\mathrm{d}w)\\
&\leq \alpha \int_\mathcal{W} \| v - v' \|_\xi \xi (f(\bm{x}, \pi(\bm{x}), w)) \hat{\bm{\mu}}(\mathrm{d} w)\\
&\leq \alpha \| v - v' \|_\xi \beta \xi (\bm{x}),
\end{split}
\end{equation}
where the last inequality holds due to Assumption~\ref{ass_sc}-1).
By switching the role of $v$ and $v'$, we also have
$(T^\pi v')(\bm{x}) - (T^\pi v)(\bm{x})  -\epsilon \leq \alpha \beta \| v - v' \|_\xi  \xi (\bm{x})$. 
Since the two inequalities hold for any $\bm{x} \in \mathcal{X}$ and $\epsilon > 0$, and $\tau = \alpha \beta$, we conclude that 
$\| T^\pi v - T^\pi v' \|_\xi \leq \tau \| v - v'\|_\xi$.
It is straightforward to check that $T^\pi$ is monotone.
\end{proof}


This lemma  implies that the value function $v^\pi$ is the unique fixed point of $T^\pi$ in $\mathbb{B}_\xi (\mathcal{X})$.
By using the contraction property of $T^\pi$ and $T$, 
we can estimate the number of iterations needed to obtain an $\epsilon$-optimal policy as follows:
\begin{proposition}\label{prop:est1}
Suppose that Assumption~\ref{ass_sc} holds.
We assume that given $\epsilon > 0$, 
the total  number of iterations, $k$, in the VI algorithm satisfies
\[
k > \frac{\log [(1-\tau)^2\epsilon] - \log ( 2b\tau) }{\log \tau },
\]
where $b \geq 0$ and $\tau \in (0,1)$ are the constants defined in Assumption~\ref{ass_sc} and Lemma~\ref{lem:cont2}, respectively.
Then, $\hat{\pi}$ obtained by the VI algorithm is an $\epsilon$-optimal policy, i.e.,
\[
\| v^{\hat{\pi}} - v^\star \|_\xi < \epsilon.
\]
\end{proposition}
\begin{proof}
By Lemma~\ref{lem:cont2} and Theorem~\ref{thm:ds}, we have
$v^{\hat{\pi}}, v_k, v^\star \in \mathbb{B}_\xi (\mathcal{X})$. 
We observe that
\begin{equation} \nonumber
\begin{split}
\| v^{\hat{\pi}}  - v^\star \|_\xi &= \| T^{\hat{\pi}} v^{\hat{\pi}}  - v^\star \|_\xi\\
&\leq \| T^{\hat{\pi}} v^{\hat{\pi}}  - T^{\hat{\pi}} v_k \|_\xi + \| T^{\hat{\pi}} v_k  - v^\star \|_\xi\\
&\leq \tau \| v^{\hat{\pi}}  -  v_k\|_\xi + \| T v_k  - T v^\star \|_\xi,
\end{split}
\end{equation}
where the last inequality holds because of Lemma~\ref{lem:cont2}, 
 $T^{\hat{\pi}} v_k = T v_k$ and $v^\star = T v^\star$. 
By Lemma~\ref{lem:contraction}, we have
\begin{equation} \label{bound1}
\begin{split}
\| v^{\hat{\pi}}  - v^\star \|_\xi &\leq \tau \| v^{\hat{\pi}}  -  v_k\|_\xi + \tau \| v_k  - v^\star \|_\xi\\
&\leq \tau \| v^{\hat{\pi}} - v^\star \|_\xi +2\tau \| v_k -  v^\star \|_\xi.
\end{split}
\end{equation}
On the other hand, by  \cite[Theorem 4.2 (a)]{Gonzalez2003},
\begin{equation} \label{bound2}
\| v_k -  v^\star \|_\xi \leq \frac{b}{1-\tau} \tau^k < \frac{1-\tau}{2\tau}\epsilon,
\end{equation}
where the second inequality holds due to the proposed choice of $k$.
Combining \eqref{bound1} and \eqref{bound2}, we conclude that 
$\| v^{\hat{\pi}} - v^\star \|_\xi < \epsilon$.
\end{proof}

A practical implementation of the VI algorithm requires a finite-state approximation such as a discretization of the state space. 
A review on such approximation methods can be found in a recent monograph~\cite{Saldi2018}.

%
%
%

%
%

\subsection{Policy Iteration}

\emph{Policy iteration} (PI) is 
an alternative way to construct an $\epsilon$-optimal policy. 
The PI algorithm can be described as follows:
\begin{enumerate}
\item Initialize $\pi_0$ as an arbitrary policy in $\Pi^{DS}$, and set $k:=0$;

\item (Policy evaluation) Find the fixed point $v^{\pi_k}$ of $T^{\pi_k}$;

\item (Policy improvement) For each $\bm{x} \in \mathcal{X}$, set 
\[
{\pi}_{k+1} (\bm{x}) := \tilde{\bm{u}},
\]
 where $\tilde{\bm{u}}$ is an optimal $\bm{u}$ of the semi-infinite program~\eqref{semi} that computes $(Tv^{\pi_k}) (\bm{x})$;
 
\item If the stopping criterion is met, then stop and set $\tilde{\pi} := \pi_{k+1}$. Otherwise, set $k \leftarrow k+1$ and go to Step 2);

\end{enumerate}
Here, the stopping criterion can be chosen as $\| v^{\pi_k} - v^{\pi_{k-1}} \|_\xi < \delta$ for a  positive constant $\delta$.
To perform the policy evaluation step (Step 2) in a computationally tractable manner, we reformulate the infinite-dimensional maximization problem in the definition of $T^\pi$ as finite dimensional by using Wasserstein DRO~\cite{Esfahani2015, Gao2016}.

\begin{proposition}\label{prop:finite}
Suppose that Assumption~\ref{ass_sc} holds and that $v \in \mathbb{B}_{\xi}(\mathcal{X})$.
Then, the operator $T^\pi: \mathbb{B}_{\xi} (\mathcal{X}) \to \mathbb{B}_{\xi} (\mathcal{X})$ satisfies
\begin{equation}\nonumber
\begin{split}
(T^\pi  v)(\bm{x}) = \sup_{{(w, q)} \in B}
 \Big [ c (\bm{x}, \pi(\bm{x})) +
  \frac{\alpha}{N} \sum_{i=1}^N  \big [ q_1 v(f(\bm{x}, \pi(\bm{x}), \underline{w}^{(i)})) + q_2 v(f(\bm{x}, \pi(\bm{x}), \overline{w}^{(i)})) \big ]\Big ],
  \end{split}
\end{equation}
where
${B} := \big \{ (\underline{w}^{(1)}, \ldots, \underline{w}^{(N)}, \overline{w}^{(1)}, \ldots, \overline{w}^{(N)}) \in \mathcal{W}^{2N}, q \in \Delta \mid \frac{1}{N} \sum_{i=1}^N  [ q_1 d( \underline{w}^{(i)}, \hat{w}^{(i)} )^p +q_ 2 d( \overline{w}^{(i)}, \hat{w}^{(i)} )^p ] \leq \theta^p \big \}$.
\end{proposition}
This proposition follows immediately  from~\cite[Corollary~2]{Gao2016}.
The optimization variables $\underline{w}^{(1)}, \ldots, \underline{w}^{(N)}$, $\overline{w}^{(1)}, \ldots, \overline{w}^{(N)}$ can be interpreted as 
the probability atoms that characterize one of the worst-case distributions.
By the contraction property of $T^{\pi_k}$ (Lemma~\ref{lem:cont2}),
we can find the fixed point $v^{\pi_k}$ of $T^{\pi_k}$ by value iteration.
In other words, we perform $v_{\tau+1} \leftarrow T^{\pi_k} v_{\tau}$, $\tau=0, 1, \ldots$, until convergence. When computing $T^{\pi_k} v_{\tau}$, we solve the finite-dimensional optimization problem in 
 Proposition~\ref{prop:finite} with $v := v_\tau$ to completely remove the infinite-dimensionality issue inherent in the definition of $T^{\pi_k}$.
In the policy improvement step, we use the semi-infinite program formulation of $T$ in Proposition~\ref{prop:si} instead of directly solving the infinite-dimensional minimax optimization problem in the definition of $T$. 
It is well known that $\lim_{k\to \infty} \| v^{\pi_k} - v^\star \|_\xi = 0$ under Assumption~\ref{ass_sc} by the monotonicity and contraction properties of  $T$ and $T^{\pi_k}$ (Lemmas~\ref{lem:contraction} and \ref{lem:cont2})~\cite[Proposition 2.5.4]{Bertsekas2012}.

However, it is usually difficult to find  the exact fixed point $v^{\pi_k}$ of $T^{\pi_k}$ in the policy evaluation step.
Thus, we propose a modified PI algorithm, which is also called \emph{optimistic policy iteration}~\cite{Puterman1978, Bertsekas2012}:
\begin{enumerate}
\item Initialize $\tilde{v}_0$ as an arbitrary function in $\mathbb{B}_{lsc} (\mathcal{X})$ and $\{M_k\}$ as a sequence of positive integers, and set $k:=1$;

\item (Policy improvement) For each $\bm{x} \in \mathcal{X}$, set 
\[
{\pi}_{k} (\bm{x}) := \tilde{\bm{u}},
\]
 where $\tilde{\bm{u}}$ is an optimal $\bm{u}$ of the semi-infinite program~\eqref{semi} that computes $(T\tilde{v}_{k-1}) (\bm{x})$;

\item (Policy evaluation) Compute
\[
\tilde{v}_{k} := (T^{\pi_{k}})^{M_k} \tilde{v}_{k-1}
\]
by solving the finite-dimensional optimization problems in Proposition~\ref{prop:finite};
 
\item If the stopping criterion is met, then stop and set $\tilde{\pi} := \pi_{k}$. Otherwise, set $k \leftarrow k+1$ and go to Step 2);

\end{enumerate}
Note that the modified PI algorithm approximately evaluates the performance of a policy $\pi_{k}$ as $\tilde{v}_{k}$ instead of finding the exact fixed point of $T^{\pi_{k}}$.
Concrete choices of the \emph{order sequence} $\{ M_k \}$ 
are discussed in \cite{Puterman2014}.
However, 
for any choice of $\{ M_k \}$, 
the modified PI algorithm converges under Assumption~\ref{ass_sc}~\cite{Bertsekas2012}:
\[
\lim_{k \to \infty} \| \tilde{v}_k -  v^\star \|_\xi = 0.
\]
As in the case of VI, 
we can estimate the number of iterations required for obtaining an $\epsilon$-optimal policy.
\begin{proposition}
Suppose that Assumption~\ref{ass_sc} holds. 
Let $r \in \mathbb{R}$ be a positive constant such that
\[
\| \tilde{v}_0 - T \tilde{v}_0 \|_\xi \leq r.
\]
We assume that given $\epsilon > 0$, the total number of iterations, $k$,  in the modified PI algorithm satisfies
\[
k \tau^k < \frac{(1-\tau)^2}{2 r} \epsilon,
\]
where $\tau \in (0,1)$ is the constant defined in Lemma~\ref{lem:cont2}.
Then, $\tilde{\pi} := \pi_k$ obtained by the modified PI algorithm is an $\epsilon$-optimal policy, i.e.,
\[
\| v^{\tilde{\pi}} - v^\star \|_\xi < \epsilon.
\]
\end{proposition}
\begin{proof}
According to Lemma~\ref{lem:cont2} and Theorem~\ref{thm:ds}, we have
$v^{\tilde{\pi}}, \tilde{v}_k, v^\star \in \mathbb{B}_\xi (\mathcal{X})$. By \cite[Lemma 2.5.4]{Bertsekas2012}, we obtain that
\[
 \tilde{v}_{k-1} - \frac{k \tau^{k-1}}{1-\tau} r \xi \leq v^\star
 \leq \tilde{v}_{k-1} + \frac{\tau^{k-1}}{1-\tau} r \xi,
\]
which implies that 
\begin{equation}\label{ineq10}
\| \tilde{v}_{k-1} - v^\star \|_\xi \leq \frac{k \tau^{k-1}}{1-\tau} r.
\end{equation}
On the other hand,  $\tilde{\pi} = \pi_k$ is a greedy policy when the value function is chosen as $\tilde{v}_{k-1}$. As in the proof of Proposition~\ref{prop:est1}, we have
$\| v^{\tilde{\pi}} - v^\star \|_\xi \leq \frac{2\tau}{1-\tau} \| \tilde{v}_{k-1} - v^\star \|_\xi$. Thus, by \eqref{ineq10},
\[
\| v^{\tilde{\pi}} - v^\star \|_\xi 
 \leq \frac{2 k \tau^k}{(1-\tau)^2}r < \epsilon,
\]
where the second inequality holds due to the proposed choice of $k$.
\end{proof}


\subsection{The Worst-Case Distribution Policy}


Given a policy $\pi \in \Pi^{DS}$ (for Player I), the worst-case distribution policy (for Player II) can be found by solving 
\[
\sup_{\gamma \in \Gamma} J_{\bm{x}} (\pi, \gamma),
\]
which is an optimal control problem. By the dynamic programming principle, the worst-case value function $v^\pi$, defined by \eqref{worst_value}, is the unique solution to the following Bellman equation:
\[
v^\pi = T^\pi v^\pi
\]
under Assumption~\ref{ass_sc}.
The worst-case value function $v^\pi$ can be computed, for example, via value iteration. 
\emph{Given $v^\pi$, how can we characterize the worst-case distribution policy?}
The following proposition indicates that, 
if the optimization problem involved in $(T^\pi v^\pi)(\bm{x})$ admits an optimal solution for all $\bm{x} \in \mathcal{X}$, then 
there exists an optimal policy for Player II, which is deterministic and stationary, and it generates a finitely-supported worst-case distribution.

\begin{proposition}[Worst-case distribution policy]\label{prop:wd}
Suppose that Assumption~\ref{ass_sc} holds, and that given $\pi \in \Pi^{DS}$
\[
 \sup_{\bm{\mu} \in \mathcal{D}} \bigg [ c (\bm{x}, \bm{u})
+ \alpha \int_{\mathcal{W}} v^{\pi}(f (\bm{x}, \pi(\bm{x}), w)) \mathrm{d}\bm{\mu}(w) \bigg ]
\]
admits an optimal solution for any $\bm{x} \in \mathcal{X}$.
Then, 
the deterministic stationary policy $\gamma^\pi: \mathcal{X} \to \mathcal{D}$ defined by
\[
\gamma^\pi (\bm{x}) := \frac{1}{2N} \sum_{i = 1}^N \big (\delta_{\underline{w}_{\bm{x}}^{\pi, (i)}}   +\delta_{\overline{w}_{\bm{x}}^{\pi, (i)}}  \big )\quad \forall \bm{x} \in \mathcal{X}
\]
is an optimal policy (for Player II) that generates a worst-case distribution for each state $\bm{x} \in \mathcal{X}$, where
$w_{\bm{x}}^\pi:= (\underline{w}_{\bm{x}}^{\pi, (1)}, \ldots, \underline{w}_{\bm{x}}^{\pi, (N)}, \overline{w}_{\bm{x}}^{\pi, (1)}, \ldots, \overline{w}_{\bm{x}}^{\pi, (N)})$ is an optimal solution of the maximization problem in Proposition~\ref{prop:finite} with $v := v^{\pi}$.
\end{proposition}
The existence of  an optimal policy, which is deterministic and stationary, follows from the dynamic programming principle when the assumptions in the proposition hold.
Thus, it is sufficient for Player II to use the same worst-case distribution for all stages. 
The structure of $\gamma^\pi (\bm{x})$ is obtained by applying \cite[Corollary 1]{Gao2016} to the maximization problem in the proposition.
Note that the worst-case distribution of this form is consistent with the discussion below Proposition~\ref{prop:finite}.
By using \cite[Corollary 2]{Gao2016}, we have the following sharper result
of characterizing the worst-case distribution with $N+1$ atoms: if the assumptions in Proposition~\ref{prop:wd} hold,
one of the worst-case distribution policies has the form 
\[
\gamma^\pi (\bm{x}) := \frac{1}{N} \sum_{i \neq i_0} \delta_{w_{\bm{x}}^{\pi, (i)}} + \frac{p_0}{N} \delta_{\underline{w}_{\bm{x}}^{\pi, (i_0)}}  + \frac{1- p_0}{N} \delta_{\overline{w}_{\bm{x}}^{\pi, (i_0)}},
\]
where $i_0 \in\{1, \ldots, N\}$, $p_0 \in [0,1]$, $\underline{w}_{\bm{x}}^{\pi, (i_0)}, \overline{w}_{\bm{x}}^{\pi, (i_0)} \in \argmin_{w \in \mathcal{W}} \{ \lambda^\star d (w, \hat{w}^{(i_0)})^p - \alpha v (f(\bm{x}, \pi (\bm{x}), w)) \}$, and ${w}_{\bm{x}}^{\pi, (i)} \in\argmin_{w \in \mathcal{W}} \{ \lambda^\star d (w, \hat{w}^{(i)})^p - \alpha v (f(\bm{x}, \pi (\bm{x}), w)) \}$ for all $i \neq i_0$. Here, $\lambda^\star$ is a dual minimizer, which must exist when the worst-case distribution exists~\cite[Corollary 1]{Gao2016}.

It is worth mentioning that Kantorovich duality and DP play a critical role in obtaining
all the results in this section. 
Based on the reformulation results and analytical properties of DR-control problems,
we demonstrate their utility in the following sections.

\section{Out-of-Sample Performance Guarantee}
\label{sec:perf}

A potential defect of the SAA-control formulation \eqref{saa_opt} is that 
its optimal policy may not perform well if a testing dataset of $w_t$ is different from the training dataset $\{\hat{w}^{(1)}, \ldots, \hat{w}^{(N)}\}$.
This issue occurs even when the testing and training datasets are sampled from the same distribution.
Such a degradation of the optimal decisions in out-of-sample tests is often called the \emph{optimizer's curse} in the  literature of decision analysis~\cite{Smith2006}. 
We show that an optimal distributionally robust policy can alleviate this issue and provide a guaranteed \emph{out-of-sample performance} if the  radius $\theta$ of Wasserstein ambiguity set is carefully determined.

Let ${\pi}_{\hat{w}}^\star \in \Pi$ denote an optimal distributionally robust policy obtained by using the training dataset $\hat{w} := \{\hat{w}^{(1)}, \ldots, \hat{w}^{(N)}\}$ of $N$ samples.
The out-of-sample performance of $\pi^\star$ is measured as
\begin{equation}\label{osp}
\mathbb{E}^{\pi_{\hat{w}}^\star}_{w_t \sim \mu} \bigg [ \sum_{t=0}^\infty \alpha^t c (x_t, u_t) \mid x_0 = \bm{x} \bigg ],
\end{equation}
which represents the expected total cost under a new sample that is generated (according to $\mu$) independent of the training dataset. 
Unfortunately, the out-of-sample performance cannot be precisely computed because the true distribution $\mu$ is unknown.
Thus, instead, we aim at establishing a \emph{probabilistic out-of-sample performance guarantee} of the form:
\begin{equation}\label{ppg}
\begin{split}
\mu^N \bigg \{
\hat{w}  \mid \: \mathbb{E}^{\pi_{\hat{w}}^\star}_{w_t \sim \mu} \bigg [ \sum_{t=0}^\infty \alpha^t c (x_t, u_t) \mid &\: x_0 = \bm{x} \bigg ]  \leq v_{\hat{w}}^\star (\bm{x}) \; \forall \bm{x} \in \mathcal{X} \bigg \} \geq 1-\beta,
\end{split}
\end{equation}
where $v^\star_{\hat{w}}$ denotes the optimal value function of the DR-control problem with the training dataset $\hat{w} := \{\hat{w}^{(1)}, \ldots, \hat{w}^{(N)}\}$, 
 and $\beta \in (0,1)$.\footnote{Here, $\hat{w}$, $\pi_{\hat{w}}^\star$ and $v_{\hat{w}}^\star$ are viewed as random objects.}
The inequality represents a bound $(1-\beta)$ on the probability that the expected cost incurred by $\pi^\star$ is no greater than the optimal value function. Note that the probability and the expected cost are  evaluated with respect to the true distribution $\mu$.
Thus, this inequality provides a probabilistic bound on the performance of $\pi^\star$ evaluated with unseen test samples drawn from $\mu$. 
Here, $v_{\hat{w}}^\star$, which depends on $\theta$, plays the role of a certificate for the out-of-sample performance.

Our goal is to identify conditions on the radius $\theta$ under which an optimal distributionally robust policy provides the probabilistic performance guarantee.
We begin by imposing the following assumption on the true distribution $\mu$:
\begin{assumption}[Light tail]\label{lt}
There exists a positive constant $q > p$ such that
\[
\rho := \int_{\mathcal{W}} \exp( \| w \|^q ) \: \mathrm{d} \mu (w) < + \infty.
\]
\end{assumption}
This assumption implies that the tail of $\mu$ decays exponentially. 
Under this condition, 
the following  measure concentration inequality holds:
\begin{theorem}[Measure concentration, Theorem 2 in \cite{Fournier2015}]\label{thm:mc}
Suppose that Assumption \ref{lt} holds. 
Let
\[
\nu_N := \frac{1}{N} \sum_{i=1}^N \delta_{\hat{w}^{(i)}}.
\]
Then,
\begin{equation} \nonumber
\begin{split}
&\mu^N \big \{\hat{w} \mid 
W_p (\mu, \nu_{N}) \geq \theta
\big \} \leq c_1 \big [ b_1(N, \theta) \bold{1}_{\{\theta\leq 1\}} + b_2(N, \theta) \bold{1}_{\{\theta > 1\}} \big ],
\end{split}
\end{equation}
where
\[
b_1 (N, \theta) := \left \{
\begin{array}{ll}
\exp (-c_2 N \theta^2) & \mbox{if } p > l/2\\
\exp  (-c_2 N (\frac{\theta}{\log(2+1/\theta)})^2  ) & \mbox{if } p = l/2\\
\exp (-c_2 N \theta^{l/p}   ) &\mbox{otherwise},
\end{array}
\right.
\]
and
\[
b_2 (N, \theta) := 
\exp ( -c_2 N \theta^{q/p} ). 
\]
Here,
$c_1, c_2$ are positive constants depending only on $l$, $q$ and $\rho$.
\end{theorem}

This theorem provides
an upper-bound of the probability that the true distribution $\mu$ lies outside of the Wasserstein ambiguity set $\mathcal{D}$.
The measure concentration inequality provides a systematic means to determine the radius  for $\mathcal{D}$ to contain the true distribution $\mu$ with probability no less than $(1-\beta)$.
As shown in the following theorem, 
the contraction property of Bellman operators enables us to extend the single-stage out-of-performance guarantee to its multi-stage counterpart with no additional requirement on $\theta$.

\begin{theorem}[Out-of-sample performance guarantee]\label{thm:osg}
Suppose that Assumptions~\ref{ass_sc} and \ref{lt} hold.
Let $\pi_{\hat{w}}^\star$ and $v_{\hat{w}}^\star$ denote an optimal policy and the optimal value function of the DR-control problem \eqref{dr_opt} with
the training dataset $\hat{w} := \{\hat{w}^{(1)}, \ldots, \hat{w}^{(N)}\}$ and
 the following Wasserstein ball radius:\footnote{This choice includes the radius proposed in~\cite{Esfahani2015} in the single-stage setting as a special case (when $p = 1$ and $l \neq 2$).}
\begin{equation} \nonumber
\begin{split}
\theta (N, \beta) :=
& \left \{
\begin{array}{ll}
\big [\frac{1}{Nc_2} \log (\frac{c_1}{\beta} ) \big ]^{{p}/{q}} & \mbox{if } N < \frac{1}{c_2} \log (\frac{c_1}{\beta}) \\
\big [\frac{1}{Nc_2} \log (\frac{c_1}{\beta} ) \big]^{{1}/{2}} 
& \mbox{if } N \geq \frac{1}{c_2} \log (\frac{c_1}{\beta}) \wedge  p>\frac{l}{2}\\
\big [\frac{1}{Nc_2} \log (\frac{c_1}{\beta} ) \big]^{{p}/{l}} 
& \mbox{if } N \geq \frac{1}{c_2} \log (\frac{c_1}{\beta})  \wedge  p<\frac{l}{2}\\
\bar{\theta}
& \mbox{if } N \geq \frac{(\log 3)^2}{c_2} \log (\frac{c_1}{\beta})  \wedge  p=\frac{l}{2},
\end{array}
\right.
\end{split}
\end{equation}
where $\bar{\theta}$ satisfies
$\frac{\bar{\theta}}{\log (2 + 1/\bar{\theta})}  = [
\frac{1}{Nc_2} \log  (\frac{c_1}{\beta}  )
]^{{1/2}}$,
and
$c_1, c_2$ are the positive constants in Theorem~\ref{thm:mc}.\footnote{The constants $c_1$ and $c_2$ in Theorem~\ref{thm:mc}
can be calculated using the proof of Theorem~2 in \cite{Fournier2015}.
However, this calculation is often conservative and thus results in a smaller radius $\theta (N, \beta)$ than necessary.
Bootstrapping and cross-validation methods can be used to reduce the conservativeness in the \emph{a priori} bound $\theta (N, \beta)$, as advocated and demonstrated in \cite{Esfahani2015}.}
Then, the probabilistic out-of-sample performance guarantee \eqref{ppg} holds.
\end{theorem}
\begin{proof}
Using Theorem~\ref{thm:mc}, we can confirm that our choice of $\theta$ provides the following probabilistic guarantee:
\begin{equation}\label{g1}
\mu^N \big \{
\hat{w}\mid  W_p(\mu, \nu_{N}) \leq \theta (N, \beta)
\big \}\geq 1-\beta.
\end{equation}
Define an operator $T^\star: \mathbb{B}_{\xi}(\mathcal{X}) \to \mathbb{B}_{\xi}(\mathcal{X})$ as
$(T^\star v)(\bm{x}) := \mathbb{E}_{\mu} [ c (\bm{x}, \pi_{\hat{w}}^\star(\bm{x})) + \alpha v(f(\bm{x}, \pi_{\hat{w}}^\star (\bm{x}), w))]$
for all $\bm{x} \in \mathcal{X}$.
It follows from \eqref{g1} that the following single-stage guarantee holds: 
\begin{equation}\label{ssg}
\mu^N \big \{
\hat{w} \mid  (T^\star {v_{\hat{w}}^\star})(\bm{x}) \leq (T {v_{\hat{w}}^\star}) (\bm{x})
\big \}\geq 1-\beta
\end{equation}
given any fixed $\bm{x} \in \mathcal{X}$.
It is straightforward to check under Assumption~\ref{ass_sc} that $T^\star$ is a monotone contraction mapping.

We now show that if $\mu \in \mathcal{D}$, then $(T^\star)^k { v_{\hat{w}}^\star} \leq   { v_{\hat{w}}^\star}$ for any $k=1, 2, \ldots$ using mathematical induction.
For $k=1$, we have $T^\star { v_{\hat{w}}^\star} \leq T{ v_{\hat{w}}^\star} = { v_{\hat{w}}^\star}$ by the minimax definition of $T$.
Suppose now that the induction hypothesis holds for some $k$. 
By the monotonicity of $T^\star$ and the definition of $T$, we have
\[
T^\star((T^\star)^k { v_{\hat{w}}^\star}) \leq T^\star   { v_{\hat{w}}^\star} \leq  T{ v_{\hat{w}}^\star} = { v_{\hat{w}}^\star},
\]
and thus the induction hypothesis is valid for $k+1$.

We now notice that
\[
\lim_{k\to \infty} ((T^\star)^k { v_{\hat{w}}^\star}) (\bm{x}) = \mathbb{E}_{w_t \sim \mu} \bigg [ \sum_{t=0}^\infty \alpha^t c (x_t, \pi_{\hat{w}}^\star(x_t)) \mid x_0 = \bm{x} \bigg ]
\]
since $T^\star$ is a contraction mapping under Assumption~\ref{ass_sc}.
Therefore, if $\mu \in \mathcal{D}$, then
\[
\mathbb{E}_{w_t \sim \mu} \bigg [ \sum_{t=0}^\infty \alpha^t c (x_t, \pi_{\hat{w}}^\star(x_t)) \mid x_0 = \bm{x} \bigg ]
 \leq v_{\hat{w}}^\star (\bm{x}) \quad \forall \bm{x} \in \mathcal{X}.
\]
By \eqref{g1},
 the probabilistic performance guarantee holds as desired.
\end{proof}
\begin{remark}
Note that the contraction property of $T$ and $T^\star$ plays a critical role in connecting the single-stage performance guarantee~\eqref{ssg} to
the multi-stage guarantee~\eqref{ppg} in a way that is independent of the number of stages.
This is a quite powerful result, because if we have a radius $\theta$ that provides a  desirable confidence level $(1-\beta)$ in the single-stage guarantee, we can use the same radius to achieve the same level of confidence in the multi-stage guarantee with no additional requirement.
\end{remark}

\section{Wasserstein Penalty Problem}
\label{sec:pen}

We now consider a slightly different version of the DR-control problem, which can be considered as a relaxation of \eqref{dr_opt} with a fixed penalty parameter $\lambda > 0$:
\[
\inf_{\pi \in\Pi} \sup_{\gamma \in \Gamma'} \;  \mathbb{E}^{\pi, \gamma} \bigg [
\sum_{t=0}^\infty \alpha^t (c (x_t, u_t) - \lambda W_p (\mu_t, \nu_N)^p )
  \mid x_0 = \bm{x} \bigg],  
\]
where the strategy space $\Gamma' := \{\gamma := (\gamma_0,  \gamma_{1}, \ldots) \:|$ $\gamma_t (\mathcal{P}(\mathcal{W}) | h_t^e) = 1 \; \forall h_t^e \in H_t^e\}$ of Player~II no longer depends on a Wasserstein ambiguity set.
Instead of using an explicit ambiguity set $\mathcal{D}$, Player II is penalized by $\lambda W_p(\mu_t, \nu_N)^p$, which can be interpreted as the cost of perturbing the empirical distribution~$\nu_N$.

\subsection{Dynamic Programming}

Under Assumption~\ref{ass_sc},
the Bellman operator ${T}'_\lambda: \mathbb{B}_{\xi}(\mathcal{X}) \to \mathbb{B}_{\xi}(\mathcal{X})$ of the Wasserstein penalty problem is defined by
\begin{equation} \nonumber
\begin{split}
&({T}'_\lambda v)(\bm{x}) := \inf_{\bm{u} \in \mathcal{U}(\bm{x})} \sup_{\bm{\mu} \in \mathcal{P}(\mathcal{W})} \mathbb{E}_{\bm{\mu}}\big [ c(\bm{x}, \bm{u}) - \lambda W_p(\bm{\mu}, \nu_N)^p + \alpha  v(f(\bm{x}, \bm{u}, w))  \big ].
\end{split}
\end{equation}
for all $\bm{x} \in \mathcal{X}$.
By using the strong duality result \cite[Theorem 1]{Gao2016},
we have the following equivalent form of $T_\lambda'$:

\begin{proposition}\label{prop:reform}
Suppose that the function $w \mapsto v(f(\bm{x}, \bm{u}, w))$ lies in $L^1 (\mathrm{d} \nu_N)$ for each $(\bm{x}, \bm{u}) \in \mathbb{K}$.
Then,
the Bellman operator ${T}'_\lambda$ can be expressed as
\begin{equation}\nonumber
\begin{split}
({T}'_\lambda v) (\bm{x}) &= \inf_{\bm{u} \in \mathcal{U}(\bm{x})} \bigg [ c (\bm{x}, \bm{u})   + \frac{1}{N} \sum_{i=1}^N \sup_{w' \in \mathcal{W}} [\alpha v (f(\bm{x}, \bm{u}, w')) - \lambda d(\hat{w}^{(i)}, w')^p ]  \bigg ]
\end{split}
\end{equation}
for all $\bm{x} \in \mathcal{X}$.
Furthermore, we have 
\begin{equation} \nonumber
\begin{split}
(T v)(\bm{x}) &= \inf_{\lambda\geq 0} \; [ (T'_\lambda v) (\bm{x}) + \lambda \theta^p] \quad \forall \bm{x} \in \mathcal{X}.
\end{split}
\end{equation}
\end{proposition}

By the results of \cite{Gonzalez2003} in the general minimax control setting, 
the optimal value function $v'$ is the unique fixed point (in $\mathbb{B}_{lsc}(\mathcal{X})$)   of $T_\lambda'$ under Assumption~\ref{ass_sc} because $T_\lambda'$ is a contraction.
We can  use value iteration to evaluate $v'$ due to the Banach fixed point theorem.
Analogous to Theorem~\ref{thm:ds}, there exists 
a deterministic stationary policy $\pi'$, which is optimal, where
$\pi'(\bm{x}) \in \argmin_{\bm{u} \in \mathcal{U}(\bm{x})}    [c (\bm{x}, \bm{u})  +
\frac{1}{N} \sum_{i=1}^N \sup_{w' \in \mathcal{W}} [\alpha v' (f(\bm{x}, \bm{u}, w')) - \lambda d(\hat{w}^{(i)}, w')^p ]  
 ]$ for all $\bm{x} \in \mathcal{X}$, under Assumption~\ref{ass_sc}.
%
%
%

\subsection{Linear-Quadratic Problem}

We now develop a solution approach, using a Riccati-type equation, to linear-quadratic (LQ) problems with the Wasserstein penalty when
\[
d(w, w')^p := \| w- w'\|^2,
\]
where $\| \cdot \|$ denotes the Euclidean norm on $\mathbb{R}^l$.
Consider a linear system of the form
\begin{equation}\label{lin_sys}
x_{t+1} = Ax_t + B u_t + \Xi w_t,
\end{equation}
where $A \in \mathbb{R}^{n \times n}$, $B \in \mathbb{R}^{n \times m}$, and $\Xi \in \mathbb{R}^{n \times l}$.
We also choose the following
 quadratic stage-wise cost function:
\begin{equation}\label{quad_cost}
c(x_t, u_t)= x_t^\top Q x_t + u_t^\top R u_t,
\end{equation}
where $Q = Q^\top \in \mathbb{R}^{n\times n}$ is positive semidefinite, and $R = R^\top \in \mathbb{R}^{m \times m}$ is positive definite.
For the sake of simplicity, we assume that 
$\mathbb{E}_{w \sim \nu_N} [ w ] = \frac{1}{N} \sum_{i=1}^N \hat{w}^{(i)} = 0$.
The case of non-zero mean is considered in Appendix~\ref{app:riccati}.
Let $\Sigma := \mathbb{E}_{w \sim \nu_N} [ w w^\top ] = \frac{1}{N} \sum_{i=1}^N \hat{w}^{(i)} (\hat{w}^{(i)})^\top$.
In the LQ setting, we also set $\mathcal{X} := \mathbb{R}^n$, $\mathcal{U}(\bm{x}) \equiv \mathcal{U} := \mathbb{R}^m$, and $\mathcal{W} := \mathbb{R}^l$.
Note that, unlike the standard LQG, the LQ problems with Wasserstein penalty do not assume that the probability distribution of random disturbances is Gaussian. In fact, the main motivation of this distributionally robust LQ formulation is to relax the assumption of Gaussian disturbance distributions in LQG, and to obtain a useful control policy  when the true distribution deviates from a Gaussian distribution.

By using DP,
we obtain the following explicit solution of the LQ problem:
\begin{theorem}\label{thm:riccati}
Suppose that 
there exists a symmetric positive semidefinite matrix $P \in \mathbb{R}^{n\times n}$ that solves the following  equation:
\begin{equation}\label{re_DR}
P = Q + \alpha A^\top P A + \alpha^2 A^\top S A
\end{equation}
with 
\begin{equation}\nonumber
\begin{split} 
S  &:=  P \Xi (\lambda I  - \alpha \Xi^\top P \Xi)^{-1} \Xi^\top P\\
& - [  I + \alpha  \Xi (\lambda I - \alpha \Xi^\top P \Xi)^{-1} \Xi^\top P ]^\top PB
 \\
&\times [R + \alpha B^\top \{ P + \alpha  P \Xi (\lambda I - \alpha \Xi^\top P \Xi)^{-1} \Xi^\top P \} B  ]^{-1}\\
&\times B^\top P [  I + \alpha  \Xi (\lambda I - \alpha \Xi^\top P \Xi)^{-1} \Xi^\top P ]
\end{split}
\end{equation}
for a sufficiently large $\lambda$.
Then, ${v}' (\bm{x}) := \bm{x}^\top P \bm{x} + z$
 solves the Bellman equation, where $z :=  \frac{\lambda}{1-\alpha} \mbox{tr}[ \{ \lambda(\lambda I - \alpha \Xi^\top P \Xi)^{-1}  - I\} \Sigma]$.
If, in addition, $v'$ is the optimal value function,\footnote{Sufficient conditions  for   $v'$ to be the optimal value function are provided in \cite{Kim2020}. Under the stabilizability and observability conditions, the algebraic Riccati equation has a unique positive semidefinite solution as well. }
then the unique optimal  policy  ${\pi'}$  is given by
\[
{\pi'} (\bm{x}) = K \bm{x}\quad \forall \bm{x} \in \mathbb{R}^n,
\]
where  
\begin{equation} \nonumber
\begin{split}
K &:= - [R + \alpha B^\top \{ P + \alpha  P \Xi (\lambda I - \alpha \Xi^\top P \Xi)^{-1} \Xi^\top P \} B  ]^{-1}\\
& \times \alpha B^\top P^\top [  I + \alpha  \Xi (\lambda I - \alpha \Xi^\top P \Xi)^{-1} \Xi^\top P ] A.
\end{split}
\end{equation}
Furthermore, if we let
\[
{w}_{\bm{x}}'^{(i)} :=  (\lambda I - \alpha \Xi^\top P \Xi )^{-1}  [\alpha \Xi^\top P (A + BK)\bm{x} + \lambda \hat{w}^{(i)}],
\]
the deterministic stationary policy $\gamma' \in \Gamma'$, defined as
\[
\gamma' (\bm{x}) =  \frac{1}{N} \sum_{i=1}^N \delta_{{w}'^{(i)}_{\bm{x}}} \quad \forall \bm{x} \in \mathbb{R}^n,
\]
is an optimal policy for Player II that generates a worst-case distribution for each $\bm{x} \in \mathbb{R}^n$.
\end{theorem}
Its proof is contained in Appendix~\ref{app:riccati}.
We first note that an optimal distributionally robust policy is \emph{linear} in the system state.
Furthermore, the control gain matrix $K$ is independent of the covariance matrix $\Sigma$ as in standard LQG. 
The worst-case distribution's support elements $w_{\bm{x}}'^{(i)}$'s are affine in the system state.
More specifically, $w_{\bm{x}}'^{(i)}$ is obtained by scaling the $i$th data sample $\hat{w}^{(i)} \in \mathbb{R}^l$ by the factor of $(\lambda I - \alpha \Xi^\top P \Xi)^{-1} \lambda$ and shifting it by the vector $(\lambda I - \alpha \Xi^\top P \Xi)^{-1} \alpha \Xi^\top P (A+BK) \bm{x}$, which is linear in the system state.
Distributional robustness is controlled by the penalty parameter $\lambda$: 
As $\lambda$ increases, the permissible deviation of $\mu_t$ from $\nu_N$ decreases. This is equivalent to decreasing the Wasserstein ball radius $\theta$ in the original DR-control setting.
Thus, by letting $\lambda$ tend to $+\infty$, the optimal distributionally robust policy for the LQ problem converges pointwise to the standard LQ optimal control policy.

\begin{proposition}
Suppose that $(A, B)$ is stabilizable and $(A, C)$ is observable, where $Q = C^\top C$.
Let $\bar{P}$ be the unique symmetric positive definite solution of 
the following discrete algebraic Riccati equation:
\begin{equation}\label{re}
\bar{P} = Q + \alpha A^\top \bar{P} A - \alpha^2 A^\top \bar{P} B (R + \alpha B^\top \bar{P} B)^{-1} B^\top \bar{P} A,
\end{equation}
and let
\[
\bar{K} := -\alpha ( R + \alpha B^\top \bar{P} B)^{-1} B^\top \bar{P} A.
\]
Then, for each $\bm{x} \in \mathcal{X}$
\begin{equation}
\begin{split}
\pi'(\bm{x}) &\to \bar{K} \bm{x} \\
w_{\bm{x}}' &\to \hat{w}_{\bm{x}}
\end{split}
\end{equation}
as $\lambda \to \infty$, where $\pi'$ and $w_{\bm{x}}'$ are defined in Theorem~\ref{thm:riccati}.
\end{proposition}
\begin{proof}
Let $P_\lambda$ denote a symmetric positive semidefinite solution of  \eqref{re_DR} given any fixed $\lambda \geq \bar{\lambda}$.
As $\lambda$ tends to $+\infty$, the right-hand side of \eqref{re_DR} tends to $Q + \alpha A^\top P_\lambda A - \alpha^2 A^\top P_\lambda B (R + \alpha B^\top P_\lambda B)^{-1} B^\top P_\lambda A$, which corresponds to the right-hand side of \eqref{re} with $\bar{P} = P_\lambda$.
Therefore, $P_\lambda$ solves the algebraic Riccati equation~\eqref{re} as $\lambda \to \infty$.
On the other hand, \eqref{re} admits a unique positive definite solution when $(A, C)$ is observable and $(A, B)$ is stabilizable~(e.g., \cite[Section~2.4]{Lewis2012}).
Thus, $P_\lambda$ converges to $\bar{P}$ as $\lambda \to \infty$.
Likewise, we can show that the feedback gain matrix $K$ and the worst-case distribution's support element $w_{\bm{x}}'^{(i)}$ (defined in Theorem~\ref{thm:riccati}) tend to $\bar{K}$ and $\hat{w}^{(i)}$, respectively, as $\lambda \to \infty$.
Therefore, the result follows.
\end{proof}

\section{Numerical Experiments} \label{sec:exp}

\subsection{Investment-Consumption Problem}

We first demonstrate the performance and utility of DR-control through an investment-consumption problem~(e.g., \cite{Samuelson1969, Hakansson1970}).
Let $x_t$ be the wealth of an investor at stage $t$. 
The investor wishes to decide the amount $u_{1,t}$ to be invested in a risky asset (with an i.i.d. random rate of return, $w_t$) and the amount $u_{2,t}$ to be consumed 
at stage $t$.
The remaining amount $(x_t - u_{1,t} - u_{2,t})$ is automatically re-invested into a riskless asset with a deterministic rate of return, $\eta$.  
Then, the investor's wealth evolves as
\[
x_{t+1} = \eta (x_t - u_{1,t} - u_{2,t}) + w_t  u_{1,t}.
\]
We assume that the control actions $u_{1,t}$ and $u_{2,t}$ satisfy
the following constraints:
\[
u_{1,t} + u_{2,t} \leq x_t, \quad u_{1,t}, u_{2,t} \geq 0 \quad \forall t, 
\]
i.e., $\mathcal{U}(\bm{x}) := \{ \bm{u} := (\bm{u}_1, \bm{u}_2) \in \mathbb{R}^2 \mid  \bm{u}_1 + \bm{u}_2 \leq \bm{x}, \bm{u} \geq 0\}$.

The cost function is given by the following negative expected utility from consumption:
\[
J(\pi, \gamma) := - \mathbb{E}^{\pi, \gamma} \bigg [
\sum_{t=0}^\infty \alpha^t U(u_{2, t})
\bigg ],
\]
where the utility function $U: \mathbb{R} \to \mathbb{R}$ is selected as $U(c) = c - \zeta c^2$.
%
The following parameters are used in the numerical simulations: $\zeta = 0.25$, $\alpha = 0.9$, $\eta = 1.02$, and $p = 1$.
The data samples $\{ \hat{w}^{(1)}, \ldots, \hat{w}^{(N)} \}$ of $w_t$  are generated according to the normal distribution $\mathcal{N}(1.08, 0.1^2)$.
We numerically approximate the optimal value function $v^\star_{\hat{w}}$ and the corresponding optimal policy $\pi^\star_{\hat{w}}$ on a computational grid 
by using the convex optimization approach in~\cite{Yang}.
This method
approximates the Bellman operator   by the optimal value of a convex program with a uniform convergence property. 
Furthermore, it does not require any explicit interpolation in evaluating the value function and control policies at some state other than the grid points, by using an auxiliary optimization variable to assign the contribution of each grid point to the next state.

The numerical experiments were conducted on a Mac with 4.2 GHz Intel Core i7 and 64GB RAM. The amount of time required for simulations with different grid sizes and $N=10$ are reported in TABLE~\ref{tab:0}.
For the rest of the simulations, we used 71 states (with grid spacing 0.02).

\begin{table}[tb]
\small
\centering
\caption{{Computation time (in seconds) for the investment-consumption problem with different grid sizes}\label{tab:0}}
\begin{tabular}{l*{10}{c}}
\# of states       & 36 & 71 & 141 & 281 \\
\hline \hline
Time (sec)  & 288. 69 & 854.61 & 2086.15 & 9350.04    \\
 \hline
\end{tabular}
\end{table}

%

\begin{figure}[tb]
\begin{center}
\includegraphics[width=5.5in]{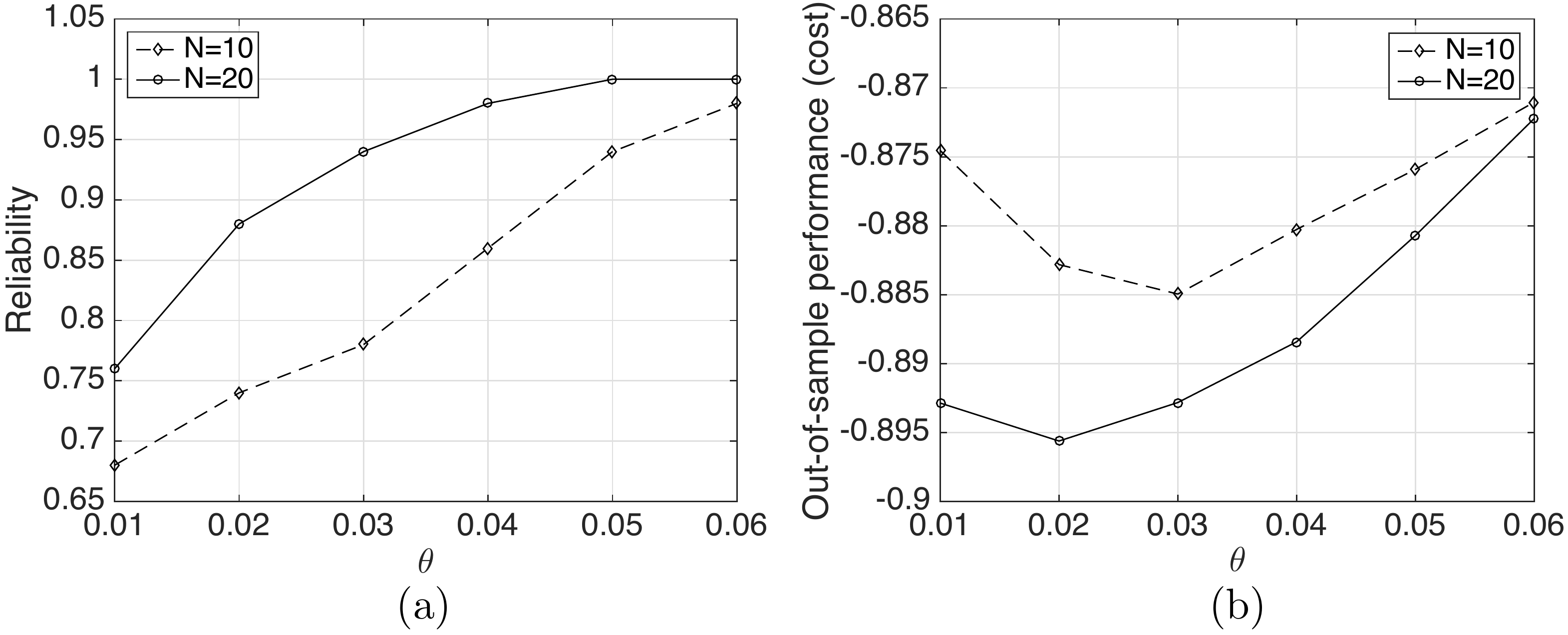}   
\caption{Depending on the radius $\theta$ and the number of samples $N$,
(a) the reliability $\mu^N \{ \hat{w} | \mathbb{E}^{\pi_{\hat{w}}^\star}_{w_t \sim \mu}  [ \sum_{t=0}^\infty \alpha^t r(x_t, u_t)|   x_0 = \bm{x}  ]  \leq v_{\hat{w}}^\star (\bm{x}) \}$, and (b)
the out-of-sample performance (cost) of $\pi_{\hat{w}}^\star$.}  
\label{fig:investment}    
\end{center}           
\end{figure}

\subsubsection{Out-of-sample performance guarantee}

To demonstrate the out-of-sample performance guarantee of an optimal distributionally robust policy, 
we compute the following \emph{reliability} of $\pi_{\hat{w}}^\star$:
\[
\mu^N \bigg \{ \hat{w} \mid \mathbb{E}^{\pi_{\hat{w}}^\star}_{w_t \sim \mu}  \bigg [ \sum_{t=0}^\infty \alpha^t c (x_t, u_t) \mid  x_0 = \bm{x}  \bigg ]  \leq v_{\hat{w}}^\star (\bm{x}) \bigg \},
\]
which represents the probability that the expected cost incurred by  $\pi_{\hat{w}}^\star$ under the true distribution $\mu$ is no greater than  $v_{\hat{w}}^\star (\bm{x})$. 
As shown in Fig. \ref{fig:investment} (a), the reliability increases with the Wasserstein ball radius $\theta$ and the number $N$ of samples. 
This result is consistent with  Theorem~\ref{thm:osg}. 
Our numerical experiments also confirm that the same radius $\theta$ can be used to achieve the same level of reliability in both single-stage and multi-stage settings as indicated in the theorem.

Fig. \ref{fig:investment} (b) illustrates the  out-of-sample cost~\eqref{osp} of $\pi_{\hat{w}}^\star$ with respect to $\theta$ and $N$.
Interestingly, the out-of-sample cost does not monotonically decrease with $\theta$.\footnote{This observation is consistent with the single-stage case in Section 7.2  of \cite{Esfahani2015}.}
For a too-small radius, the resulting DR-policy is not sufficiently robust to obtain the best out-of-sample performance (i.e., the least out-of-sample cost).
On the other hand, if a too-large Wasserstein ambiguity set is selected, 
the resulting DR-policy is overly conservative and thus sacrifices the closed-loop performance. 
Thus, there exists an optimal radius (e.g., $0.02$ in the case of $N = 20$) that provides the best out-of-sample performance.

\subsubsection{Comparison to SAA}
To compare DR-control~\eqref{dr_opt} with SAA-control~\eqref{saa_opt},
we first compute the out-of-sample performance of $\pi_{\hat{w}}^\star$ and that of the corresponding optimal SAA policy $\pi_{\hat{w}}^{\tiny \mbox{SAA}}$ obtained by using the same training dataset $\hat{w}$.
The  radius is selected as the one that provides the best out-of-sample performance.
As shown in Fig.~\ref{fig:investment_SAA}, the proposed DR-policy achieves   8\% lower out-of-sample cost than the SAA-policy when $N = 10$. 
As expected, the gap between the two decreases with the number of samples.
Note that the proposed DR-policy designed  even  with a small number of samples ($N=10$)
maintains its performance under the test dataset that is generated independent of the training dataset,
unlike the corresponding SAA-policy.

\begin{figure}[tb]
\begin{center}
\includegraphics[width=2.7in]{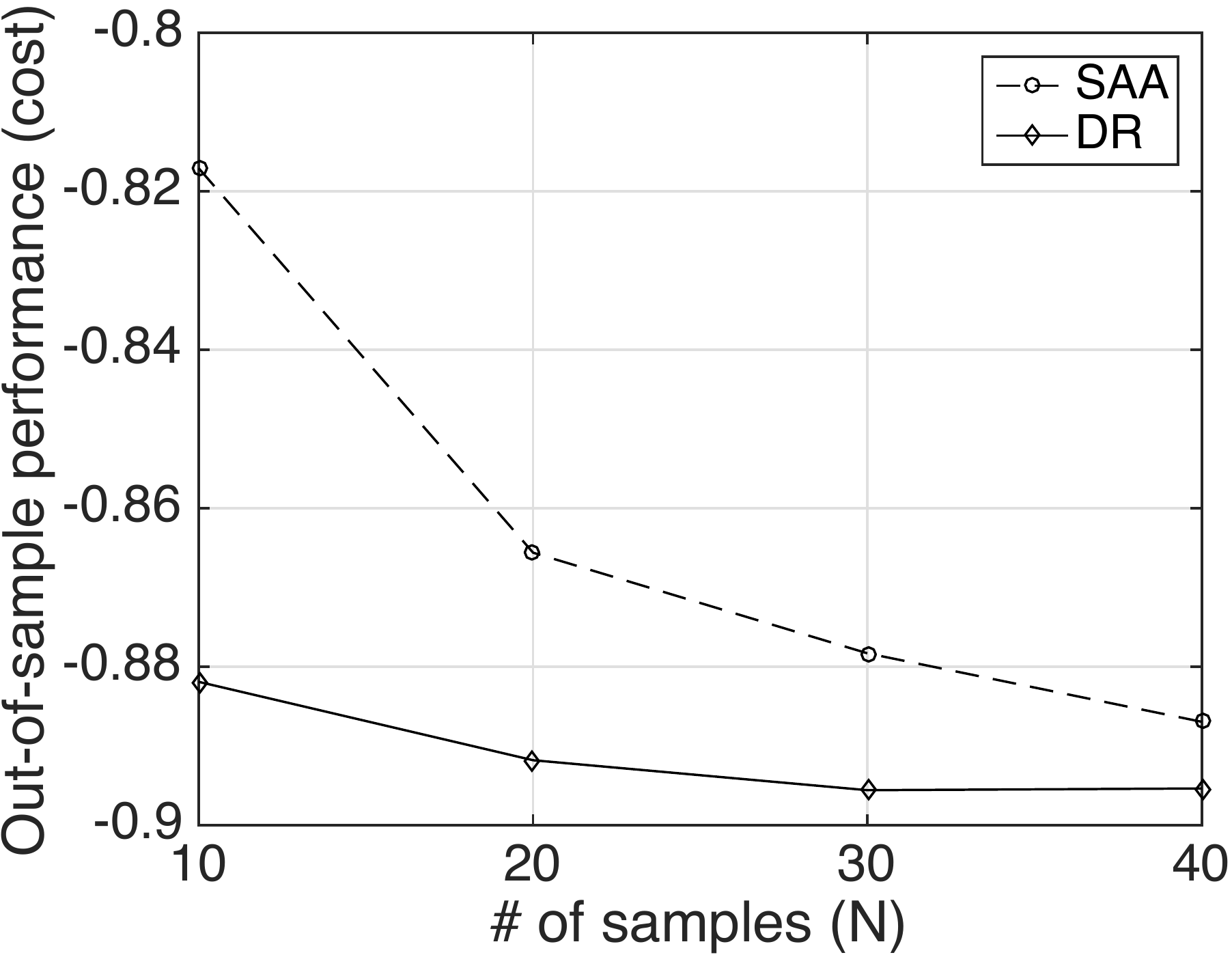}   
\caption{The out-of-sample performance (cost) of the optimal SAA policy $\pi_{\hat{w}}^{\tiny \mbox{SAA}}$ ($\circ$) and the optimal distributionally robust policy $\pi_{\hat{w}}^\star$ ($\diamond$) depending on $N$.}  
\label{fig:investment_SAA}    
\end{center}           
\end{figure}

%

\subsection{Power System Frequency Control Problem} 

Consider an electric power transmission system with $N$ buses (and $\bar{n}$ generator buses).
This system may be subject to ambiguous uncertainty generated from variable renewable energy sources such as wind and solar.
For the frequency regulation of this  system, 
we use the proposed Wasserstein penalty method to control
the mechanical power input of generator.
Let $\bm{\theta}_i$ and $P_{e, i}$ be the voltage angle (in radian) and the mechanical power input (in per unit), respectively, at generator bus $i$.
The swing equation of this system is then given by
\begin{equation}\label{swing}
M_i \ddot{\bm{\theta}}_i (t) + D_i \dot{\bm{\theta}}_i (t) = P_{m,i} (t) - P_{e, i} (t) \quad \forall i = 1, \ldots, \bar{n},
\end{equation}
where $M_i$ and $D_i$ denote the inertia coefficient (in pu$\cdot$sec$^2$/rad) and the damping coefficient (in pu$\cdot$sec/rad) of the generator at bus $i$.
Here, $P_{e,i}$ is the electrical active power injection (in per unit) at bus $i$ and is given by
$P_{e, i} := \sum_{j=1}^N | V_i| |V_j |  (G_{ij} \cos ({\bm \theta}_i - {\bm \theta}_j) + B_{ij} \sin ( {\bm \theta}_i - {\bm \theta}_j) )$,
where $G_{ij}$ and $B_{ij}$ are the conductance and susceptance of the transmission line connecting buses $i$ and $j$, respectively, and $V_i$ is the voltage at bus $i$.
Assuming that all the voltage magnitudes are $1$ per unit, the angle differences $|\bm{\theta}_i - \bm{\theta}_j|$'s are small, and  
all the transmission lines are (almost) lossless,
the AC power flow equation can be approximated by the following linearized DC power flow equation:
\begin{equation}\label{dc}
P_{e, i} := \sum_{j=1}^N B_{ij} ({\bm \theta}_i - {\bm \theta}_j) \quad \mbox{or} \quad P_{e} = L \bm{\theta},
\end{equation}
where $P_e := (P_{e,1}, \ldots, P_{e, \bar{n}})$,
$\bm{\theta} := (\bm{\theta}_{1}, \ldots, \bm{\theta}_{\bar{n}})$, and 
$L \in \mathbb{R}^{\bar{n} \times \bar{n}}$ is the Kron-reduced Laplacian matrix of this power network.\footnote{The Kron reduction is used to express the system in the reduced dimension $\bar{n}$ by focusing on the interactions of the generator buses~\cite{Bergen1999}. 
More precisely, we can obtain the Kron-reduced admittance matrix $Y^{\mbox{\tiny Kron}}$, by eliminating nongenerator bus $k$, as
$Y_{ij}^{\mbox{\tiny Kron}} := Y_{ij} - Y_{ik} Y_{kj} / Y_{kk}$ for all $i, j = 1, \ldots, N$ such that $i, j\neq k$.
The Kron-reduced Laplacian can then be obtained by setting
$L_{ii} := \sum_{k=1, \ldots, \bar{n} : k\neq i} B_{ik}^{\mbox{\tiny Kron}}$ and 
$L_{ij} := -B_{ij}^{\mbox{\tiny Kron}}$ for $i \neq j$ , where $B^{\mbox{\tiny Kron}}$ denotes the susceptance of the Kron-reduced admittance matrix~\cite{Dorfler2013}.
} 

Let $x(t) := ( \bm{\theta}(t)^\top, \dot{\bm{\theta}}(t)^\top )^\top \in \mathbb{R}^{2\bar{n}}$ and $u(t) := P_{m} (t) \in \mathbb{R}^{\bar{n}}$.
By combining \eqref{swing} and \eqref{dc}, we obtain the following state-space model of the power system~(e.g., \cite{Fazelnia2017}):
\[
\dot{x} (t)
= 
\begin{bmatrix}
0 & I \\
-M^{-1} L & M^{-1} D
\end{bmatrix}
x(t)
+
\begin{bmatrix}
0\\
M^{-1}
\end{bmatrix}
u(t),
\]
where $M := \mbox{diag}(M_1, \ldots, M_{\bar{n}})$ and
$D := \mbox{diag}(D_1, \ldots, D_{\bar{n}})$.
We discretize this system using zero-order hold on the input and a sampling time of $0.1$ seconds
to
obtain the matrices $A$ and $B$ of the following discrete-time system model~\eqref{lin_sys}:
\[
x_{t+1} = A x_t + B (u_t + w_t)
\]
where $w_{i,t}$ is the random disturbance (in per unit) at bus $i$ at stage $t$.
It can model uncertain power injections generated by solar or wind energy sources.

The state-dependent portion of the quadratic cost function \eqref{quad_cost} is chosen as
\[
x^\top Q x := 
\bm{\theta}^\top [I -  \bold{1} \bold{1}^\top/\bar{n} ] \bm{\theta} + \frac{1}{2} \dot{\bm\theta}^\top M \dot{\bm\theta},
\]
where $\bold{1}$ denotes the $\bar{n}$-dimensional vector of all ones, the first term measures the deviation of rotor angles from their average $\bar{\bm{\theta}} := \bold{1}^\top \bm{\theta}/\bar{n}$, and 
the second term corresponds to the kinetic energy stored in the electro-mechanical generators~\cite{Dorfler2014}.
The matrix $R$ is chosen to be the $\bar{n}$ by $\bar{n}$ identity matrix.

The IEEE 39-bus New England test case (with 10 generator buses, 29 load buses, and 40 transmission lines)  is used to demonstrate
the performance of the proposed LQ control $\pi_{\hat{w}}'$ with Wasserstein penalty.
The initial values of voltage  angles $\bm{\theta}(0)$ are determined by solving the (steady-state) power flow problem using MATPOWER~\cite{Zimmerman2011}.
The initial frequency is set to be zero for all buses except bus 1 at which $\dot{\bm{\theta}}_1 (0) := 0.1$ per unit.
We use $\alpha = 0.9$  in all simulations.

\begin{figure}[tb]
\begin{center}
\includegraphics[width=5.5in]{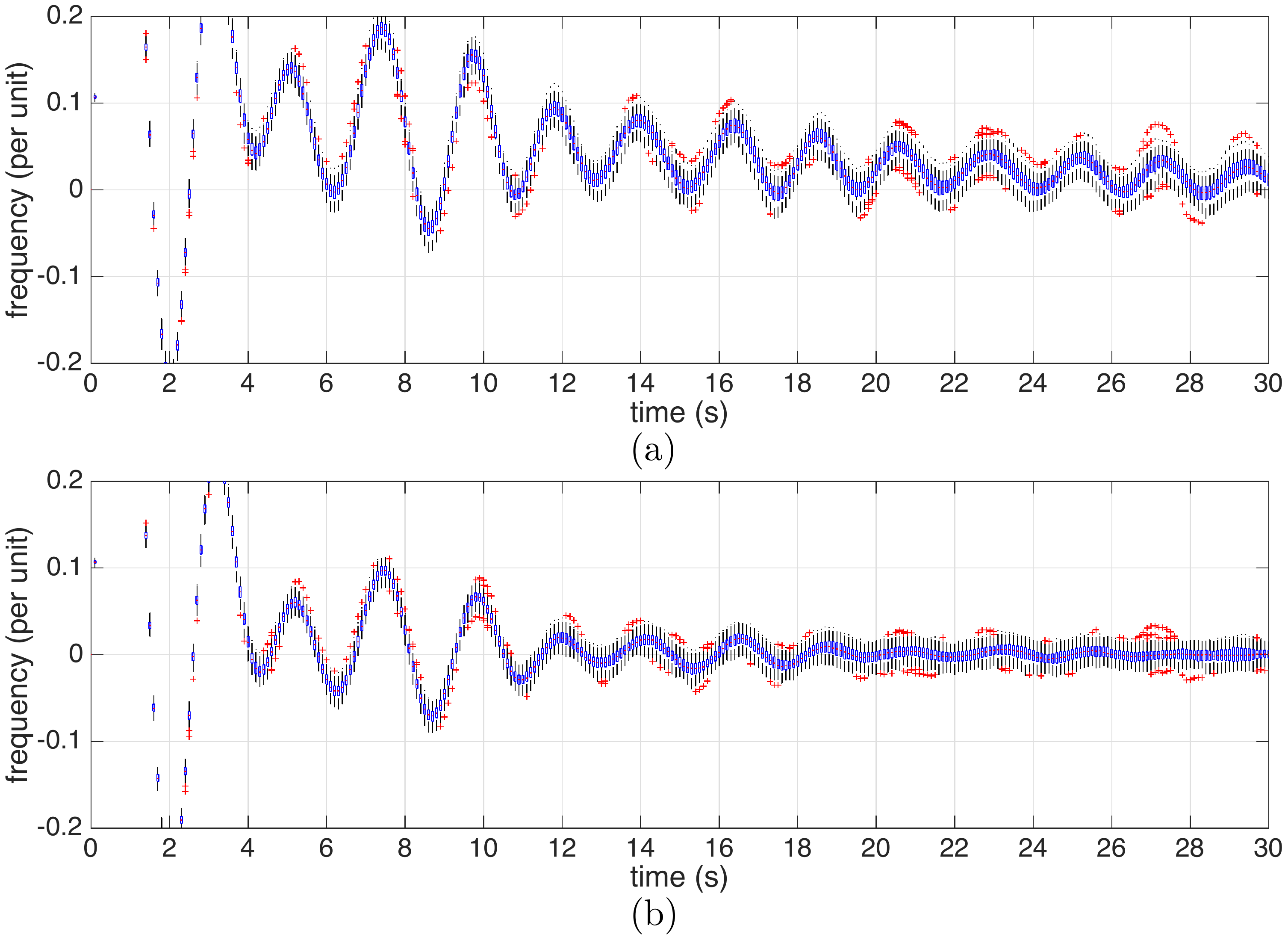}   
\caption{The box plot
of frequency deviation $\dot{\bm{\theta}}_{10}$ controlled by (a) the standard LQG control policy $\pi_{\hat{w}}^{\mathrm{\tiny LQG}}$, and (b) the optimal DR-control policy $\pi_{\hat{w}}'$ with Wasserstein penalty, under the worst-case distribution policy. }  
\label{fig:freq}    
\end{center}           
\end{figure}

\begin{table*}[tb]
\small
\centering
\caption{The amount of time (in seconds) required to decrease and maintain the mean frequency deviation less than $1\%$}\label{tab:1}
\begin{tabular}{l*{10}{c}}
Bus             & 1 & 2 & 3 & 4 & 5  & 6 & 7 &8 &9 &10\\
\hline \hline
$\pi^{\mathrm{\tiny LQG}}_{\hat{w}}$ & 73.5 & 70.3 & 59.3 & 21.5 & 21.5 & 24.2 & 21.3 & 62.5 & 36.5 &27.7   \\
$\pi'_{\hat{w}}$        & 25.0 & 24.2 & 19.8 & 12.4 &  12.3 & 11.6 & 12.2 & 20.8 & 14.3 &  14.3  \\
 \hline
\end{tabular}
\end{table*}

\subsubsection{Worst-case distribution policy}
We first compare the standard LQG control policy $\pi_{\hat{w}}^{\mathrm{\tiny LQG}}$ and
the proposed DR-control policy $\pi_{\hat{w}}'$ with the Wasserstein penalty under the worst-case distribution policy~$\gamma_{\hat{w}}'$ obtained by using the proof of  Theorem~\ref{thm:riccati}.
We set $N = 10$ and $\lambda = 0.03$.
The i.i.d. samples $\{\hat{w}^{(i)}\}_{i=1}^N$ are generated according to the normal distribution $\mathcal{N}(0, 0.1^2I)$.
As depicted in Fig.~\ref{fig:freq},\footnote{The central bar on each box indicates the median; the bottom and top edges of the box indicate the 25th and 75th percentiles, respectively; and the `+' symbol represents the outliers.} $\pi_{\hat{w}}'$ is less sensitive than $\pi_{\hat{w}}^{\mathrm{\tiny LQG}}$ against the worst-case distribution policy.\footnote{The frequency deviation at other buses displays a similar behavior.}
In the $[0, 24]$ (seconds) interval, the frequency controlled by $\pi_{\hat{w}}^{\mathrm{\tiny LQG}}$ fluctuates around non-zero values while $\pi_{\hat{w}}'$ 
maintains the frequency fluctuation centered approximately around zero.
This is because the proposed DR-method takes into account the possibility of nonzero-mean disturbances, while the standard LQG method assumes zero-mean disturbances.
Furthermore, the proposed DR-method suppress the frequency fluctuation much faster than the standard LQG method: Under $\pi_{\hat{w}}'$, the mean frequency deviation averaging across the buses is less than 1\% for any time after 16.7 seconds. 
On the other hand, if the standard LQG control is used, 
it takes 41.8 seconds to take 
the mean frequency deviation (averaging across the buses) below 1\%.
The detailed results for each bus are reported in Table~\ref{tab:1}.


\subsubsection{Out-of-sample performance guarantee}

We now examine the out-of-sample performance of $\pi_{\hat{w}}'$ 
and how it depends on the penalty parameter $\lambda$ and the number $N$ of samples.
The i.i.d. samples $\{\hat{w}^{(i)}\}_{i=1}^N$ are generated according to the normal distribution $\mathcal{N}(0, I)$.
Given $\lambda$ and $N$, we define the \emph{reliability} of $\pi_{\hat{w}}'$ as 
\[
\mu^N \bigg \{ \hat{w} \mid \mathbb{E}^{\pi_{\hat{w}}'}_{w_t \sim \mu}  \bigg [ \sum_{t=0}^\infty \alpha^t c (x_t, u_t) \mid  x_0 = \bm{x}  \bigg ]  \leq v_{\hat{w}}' (\bm{x}) \bigg \}.
\]
As shown in Fig.~\ref{fig:power_reliability},  
the reliability decreases with $\lambda$.
This is because when using larger $\lambda$,
 the control policy $\pi_{\hat{w}}'$ becomes less robust against the deviation of the empirical distribution from the true distribution.
Increasing $\lambda$ has the effect of decreasing the radius $\theta$ in DR-control.
In addition, the reliability tends to increase as the number $N$ of samples used to design $\pi_{\hat{w}}'$
increases.
This result is consistent with the dependency of the DR-control reliability on the number of samples.
By using this result, we can determine the penalty parameter to attain a desired out-of-sample performance guarantee (or reliability), given the number of samples.

\begin{figure}[tb]
\begin{center}
\includegraphics[width=2.7in]{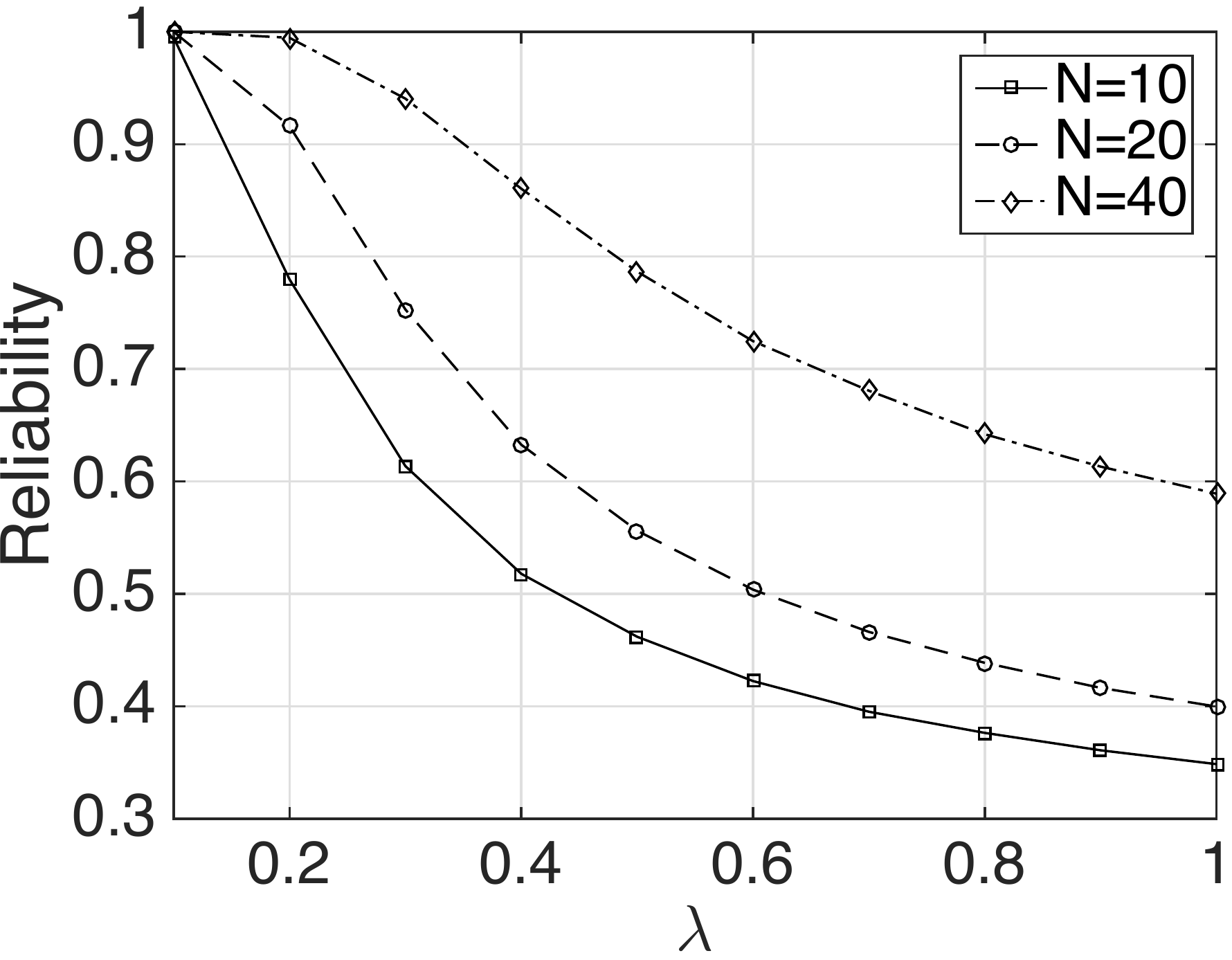}   
\caption{The reliability $\mu^N \{ \hat{w} \mid \mathbb{E}^{\pi_{\hat{w}}'}_{w_t \sim \mu}  [ \sum_{t=0}^\infty \alpha^t c (x_t, u_t)|   x_0 = \bm{x}  ]  \leq v_{\hat{w}}' (\bm{x}) \}$, in the Wasserstein penalty case, depending on $\lambda$ and $N$.}  
\label{fig:power_reliability}    
\end{center}           
\end{figure}

\section{Conclusions}

In this paper, we considered distributionally robust stochastic control problems with Wasserstein ambiguity sets by directly using the data samples of uncertain variables. 
We showed that the proposed framework has several salient features, including $(i)$ computational tractability with error bounds, $(ii)$ an out-of-sample performance guarantee, and $(iii)$ an explicit solution in the LQ setting.
It is worth emphasizing that the Kantorovich duality principle plays a critical role in our DP solution and analysis.
Furthermore, with regard to the out-of-sample performance guarantee, 
our analysis provides the unique insight that
the contraction property of the Bellman operators
extends a single-stage guarantee---obtained using a measure concentration inequality---to the corresponding multi-stage guarantee without any degradation in the confidence level.



\appendix
\section{Proof of Lemma~\ref{lem:ball}}\label{app:ball}

\begin{proof}
Recall that
using the Kantorovich duality principle, the Wasserstein distance between ${\mu}$ and $\nu$ can be written as
\[
W(\mu, \nu_N) = \sup_{\varphi, \psi \in \Phi} \bigg \{
\int_{\mathcal{W}} \varphi (w) \: \mathrm{d} \mu (w)  + 
\int_{\mathcal{W}} \psi (w') \: \mathrm{d} \nu_N (w') 
\bigg \},
\]
where $\Phi := \{ 
(\varphi, \psi) \in L^1 (\mathrm{d} \mu) \times L^1(\mathrm{d} \nu_N) \mid
\varphi (w) + \psi (w') \leq d( w, w')^p \; \forall w, w' \in \mathcal{W}
\}$.
Let 
\begin{equation} \nonumber
\begin{split}
&\hat{\mathcal{D}} := \bigg \{ {\mu} \in \mathcal{P} (\mathcal{W}) \bigg | 
\int_{\mathcal{W}} \varphi (w) \: \mathrm{d} {\mu} (w) \: + \int_{\mathcal{W}} \inf_{w \in \mathcal{W}} [ d ( w,  w')^p  - \varphi (w) ] \: \mathrm{d}\nu_N (w')  \leq \theta^p \:\: \forall \varphi \in L^1(\mathrm{d} {\mu})
\bigg \}.
\end{split}
\end{equation}
We claim that $\hat{\mathcal{D}} = \mathcal{D}$.
Choose an arbitrary ${\mu}$ from $\hat{\mathcal{D}}$.
Note that for any $(\varphi, \psi) \in \Phi_d$, 
\[
\psi (w') \leq \inf_{w \in \mathcal{W}} [ d( w, w')^p  -\varphi (w)] \quad  \forall w' \in \mathcal{W}.
\]
Thus, we have
\begin{equation} \nonumber
\begin{split}
W({\mu}, \nu_N) &\leq \sup_{\varphi \in L^1(\mathrm{d} {\mu})} \bigg \{
\int_{\mathcal{W}} \varphi(w) \: \mathrm{d}{ \mu}(w) + \int_{\mathcal{W}}\inf_{w \in \mathcal{W}} [  d( w, w' )^p  -\varphi (w)] \: \mathrm{d}\nu_N (w')
\bigg \} \leq \theta^p,
\end{split}
\end{equation}
where the last inequality holds becase ${\mu} \in \hat{\mathcal{D}}$. Therefore, ${\mu} \in \mathcal{D}$, which implies that $\hat{\mathcal{D}} \subseteq \mathcal{D}$.

We now select an arbitrary ${\mu}$ from $\mathcal{D}$.
Fix  $\varphi \in L^1(\mathrm{d}\bm{\mu})$ and
define a function $\hat{\psi}: \mathcal{W} \to \mathbb{R}$ by
\[
\hat{\psi} (w') := \inf_{w \in \mathcal{W}} [ d( w,  w')^p - \varphi (w)] \quad \forall  w' \in \mathcal{W}.
\]
Then, $\hat{\psi} \in L^1(\mathrm{d}{\mu})$ and $(\varphi, \hat{\psi}) \in \Phi$. Thus, 
\[
\int_{\mathcal{W}} \varphi(w) \: \mathrm{d} {\mu} (w) +
\int_{\mathcal{W}} \hat{\psi}(w') \: \mathrm{d} \nu_N (w')
\leq W ({\mu}, \nu_N) \leq \theta^p,
\]
which holds for any $\varphi \in L^1(\mathrm{d}{\mu})$.
By the definition of $\hat{\psi}$, this implies that
${\mu} \in \hat{\mathcal{D}}$. Therefore, $\mathcal{D} \subseteq \hat{\mathcal{D}}$.
\end{proof}

\section{Linear-Quadratic Problems}\label{app:riccati}

\begin{proof}[Proof of Theorem~\ref{thm:riccati}]
Let $v': \mathbb{R}^n \to \mathbb{R}$ be defined as
$v' (\bm{x}) := \bm{x}^\top P \bm{x} + z$.
To compute ${T}_\lambda' v'$,
we first calculate the inner maximization part  in Proposition~\ref{prop:reform} as
follows:
\begin{equation} \nonumber
\begin{split}
\phi (\bm{u}, w) &:= \sup_{w' \in \mathbb{R}^l} \big [\alpha v (f(\bm{x}, \bm{u}, w')) - \lambda d(w, w')^p \big ]\\
& = \sup_{w' \in \mathbb{R}^l} \big [  \alpha (A\bm{x} + B\bm{u} + \Xi w')^\top P (A\bm{x} + B\bm{u} + \Xi w')  + \alpha z - \lambda \| w - w'\|^2 \big ]. 
\end{split}
\end{equation}
There exists a constant $\bar{\lambda} > 0$ (depending on $P$) such that
for any $\lambda \geq \bar{\lambda}$, the objective function of the maximization problem above is strictly concave in $w'$ (i.e., $\lambda I - \alpha\Xi^\top P \Xi$ is positive definite), and thus the unique maximizer is given by
\begin{equation}\label{eq:wd}
w^\star := (\lambda I - \alpha \Xi^\top P \Xi )^{-1}  [\alpha \Xi^\top P (A\bm{x} + B\bm{u}) + \lambda w].
\end{equation}
With this maximizer, we can rewrite the term $\phi (\bm{u}, w)$ as
\begin{equation}\nonumber
\begin{split}
\phi (\bm{u}, w) &= \alpha [\bm{x}^\top  A^\top P A\bm{x} + \bm{u}^\top  B^\top P B  \bm{u} +2 \bm{x}^\top A^\top P B \bm{u} + z]\\
&+ [\alpha \Xi^\top P (A\bm{x} + B\bm{u}) + \lambda w]^\top (\lambda I - \alpha \Xi^\top P \Xi )^{-1} [\alpha \Xi^\top P (A\bm{x} + B\bm{u}) + \lambda w]-\lambda \|w\|^2.
\end{split}
\end{equation}
Since $\mathbb{E}_{w \sim \nu_N} [w ] = 0$ and $\mathbb{E}_{w \sim \nu_N} [w w^\top ] = \Sigma$, we have
\begin{equation} \nonumber
\begin{split}
\mathbb{E}_{w \sim \nu_N} [ \phi(\bm{u}, w) ] =
& \bm{u}^\top [\alpha B^\top P B + \alpha^2 B^\top P \Xi (\lambda I - \alpha \Xi^\top P \Xi)^{-1} \Xi^\top P B ]\bm{u}\\
&+ 2\alpha[ \bm{x}^\top A^\top + \alpha \bm{x}^\top A^\top P \Xi (\lambda I - \alpha \Xi^\top P \Xi)^{-1} \Xi^\top  ] P B \bm{u}\\
&+  \bm{x}^\top [ \alpha A^\top P A+ \alpha^2  A^\top P \Xi (\lambda I  - \alpha \Xi^\top P \Xi)^{-1} \Xi^\top P A] \bm{x}\\
&+\alpha z + \lambda^2 \mathrm{tr}[ (\lambda I - \alpha \Xi^\top P \Xi)^{-1}  \Sigma]
-\lambda \mathrm{tr}[\Sigma].
\end{split}
\end{equation}
Recall that 
\begin{equation}\label{obj}
(T' v')(\bm{x}) = \inf_{\bm{u} \in \mathcal{U}(\bm{x})} \big [ r(\bm{x}, \bm{u}) 
+ \mathbb{E}_{w \sim \nu_N} [ \phi(\bm{u}, w) ]  \big ].
\end{equation}
We notice 
$R + \alpha B^\top P B + \alpha^2 B^\top P \Xi (\lambda I - \alpha \Xi^\top P \Xi)^{-1} \Xi^\top P B$
is positive definite for $\lambda \geq \bar{\lambda}$ because $R$ is positive definite and $\lambda I - \alpha\Xi^\top P \Xi$ is positive definite for $\lambda \geq \bar{\lambda}$.
Thus, the objective function in \eqref{obj} is strictly convex  in $\bm{u}$ and has the unique minimizer
$\bm{u}^\star =  K\bm{x}$.  
Therefore, we obtain that
\begin{equation} \nonumber
\begin{split}
({T}' v')(\bm{x}) &=  \bm{x}^\top (Q + \alpha A^\top P A+ \alpha^2  A^\top S A )\bm{x}   +\alpha z + \lambda \mbox{tr}[ \{ \lambda(\lambda I - \alpha \Xi^\top P \Xi)^{-1}  - I\} \Sigma].
\end{split}
\end{equation}
We conclude that $v'$ solves the Bellman equation
since $P$ and $z$ satisfy 
$P = Q + \alpha A^\top P A + \alpha^2 A^\top S A$ and
$(1-\alpha) z =  \lambda \mbox{tr}[ \{ \lambda(\lambda I - \alpha \Xi^\top P \Xi)^{-1}  - I\} \Sigma]$.
Furthermore, when $v'$ is the optimal value function, 
the value of an optimal policy ${\pi}'$ at $\bm{x} \in \mathbb{R}^n$ is uniquely given by $\bm{u}^\star$, i.e., ${\pi}'(\bm{x}) = K \bm{x}$.

We now characterize the  worst-case distribution policy. Plugging $w = \hat{w}^{(i)}$ and $\bm{u} = K\bm{x}$ into \eqref{eq:wd}, we obtain that
\[
w_{\bm{x}}'^{(i)} = (\lambda I - \alpha \Xi^\top P \Xi )^{-1}  [\alpha \Xi^\top P (A\bm{x} + BK \bm{x}) + \lambda \hat{w}^{(i)}].
\]
Let $\gamma' (\bm{x}) :=  \frac{1}{N} \sum_{i=1}^N \delta_{{w}'^{(i)}_{\bm{x}}}$ for all $\bm{x} \in \mathcal{X}$.
Then, 
\begin{equation}\nonumber
\begin{split}
&W_2(\gamma' (\bm{x}), \nu_N)^2 \\
&= \min \Big \{\sum_{i, j=1}^N \kappa_{i,j} \| w'^{(i)}_{\bm{x}} - \hat{w}^{(j)} \|^2 \mid
\sum_{j=1}^N \kappa_{i,j} = \frac{1}{N}, i=1, \ldots, N, \;
\sum_{i=1}^N \kappa_{i,j} =  \frac{1}{N}, j=1, \ldots, N
\Big \}\\
&\leq \frac{1}{N} \sum_{i=1}^N  \| w'^{(i)}_{\bm{x}} - \hat{w}^{(i)} \|^2.
\end{split}
\end{equation}
Therefore, 
we have
\begin{equation} \nonumber
\begin{split}
&\mathbb{E}_{\gamma' (\bm{x})} \big [ c(\bm{x}, \bm{u}^\star) -\lambda W_2(\gamma' (\bm{x}), \nu_N)^2 - \alpha v (f(\bm{x}, \bm{u}^\star, w)) \big ] \\
&\geq c(\bm{x}, \bm{u}^\star) -\lambda \sum_{i=1}^N \| {w}'^{(i)}_{\bm{x}}- \hat{w}^{(i)}\|^2  -\frac{\alpha}{N} \sum_{i=1}^N v (f(\bm{x}, \bm{u}^\star, {w}'^{(i)}_{\bm{x}})) \\
&= c(\bm{x}, \bm{u}^\star)  + \frac{1}{N} \sum_{i=1}^N \sup_{w \in \mathbb{R}^l} \big [ \alpha v (f (\bm{x}, \bm{u}^\star, w)) - \lambda \| {w}- \hat{w}^{(i)}\|^2 \big ],
\end{split}
\end{equation}
where the last equality holds by the definition of $w'^{(i)}_{\bm{x}}$'s.
On the other hand, it follows from Proposition~\ref{prop:reform} that
\begin{equation}\nonumber
\begin{split}
&\sup_{\bm{\mu} \in \mathcal{P}(\mathbb{R}^l)}  \big [ c(\bm{x}, \bm{u}^\star) -\lambda W_2(\bm{\mu}, \nu_N)^2 - \alpha v (f(\bm{x}, \bm{u}^\star, w)) \big ] \\
&= c(\bm{x}, \bm{u}^\star)  + \frac{1}{N} \sum_{i=1}^N \sup_{w \in \mathbb{R}^l} \big [ \alpha v (f (\bm{x}, \bm{u}^\star, w)) - \lambda \| {w}- \hat{w}^{(i)}\|^2 \big ]. 
\end{split}
\end{equation}
Thus, we conclude that $\gamma' (\bm{x})$ is one of the worst-case distributions. 
\end{proof}

We now consider the case in which the data samples $\hat{w}^{(i)}$'s have non-zero mean, i.e.,
\[
\bar{w} := \mathbb{E}_{w \sim \nu_N} [ w] = \frac{1}{N} \sum_{i=1}^N \hat{w}^{(i)} \neq 0.
\]
The linear system~\eqref{lin_sys} can be rewritten as
\[
x_{t+1} = A x_t + B u_t  + \Xi w_t' + \Xi \bar{w},
\]
where $w_t' := w_t - \bar{w}$.
We now normalize the data samples $\hat{w}'^{(i)} := \hat{w}^{(i)} - \bar{w}$ for all $i \in \mathcal{I}$ so that 
\[
\frac{1}{N} \sum_{i=1}^N \hat{w}'^{(i)} = 0.
\]
Let $\bar{x} := (I-A)^{-1} \Xi \bar{w}$ assuming it is well-defined. Then,
\[
\begin{bmatrix}
x_{t+1} - \bar{x}\\
1
\end{bmatrix}
 = 
 \begin{bmatrix}
 A & 0\\
 0 & 1
 \end{bmatrix}
 \begin{bmatrix}
x_{t} - \bar{x}\\
1
\end{bmatrix}
+ 
\begin{bmatrix}
B \\
0
\end{bmatrix}
u_t
+ 
\begin{bmatrix}
\Xi \\
0
\end{bmatrix}
w_t'.
\]
By letting $x_t' := ((x_{t+1} - \bar{x})^\top, 1)^\top \in \mathbb{R}^{n+1}$, we can rewrite the system as
\[
x_{t+1}' = A' x_t + B' u_t + \Xi' w_t'.
\]
Define a positive semidefinite matrix $Q' \in \mathbb{R}^{(n+1)\times(n+1)}$ by
\[
Q' := \begin{bmatrix}
I & \bar{x}
\end{bmatrix}^\top  Q \begin{bmatrix}
I & \bar{x}
\end{bmatrix} =
\begin{bmatrix}
Q & Q\bar{x}\\
\bar{x}^\top Q & \bar{x}^\top  Q \bar{x}
\end{bmatrix}.
\]
We then have
\begin{equation} \nonumber
\begin{split}
x_t^\top Q x_t 
&= x_t'^\top Q' x_t'.
\end{split}
\end{equation}
Thus, the nonzero mean case is converted to the zero mean case with the normalized data $\hat{w}'^{(i)}$'s, the expanded state $x_t'$ and the new positive semidefinite matrix $Q'$ in the quadratic cost function.
Therefore, we can use Theorem~\ref{thm:riccati} to compute the DR-control gain matrix $K'$.
The corresponding optimal policy is obtained as
$\pi'(\bm{x}) := K' (
(\bm{x} - \bar{x})^\top,
1)^\top$ for all $\bm{x} \in \mathbb{R}^n$.

\bibliographystyle{IEEEtran}

\bibliography{reference}

\end{document}